\RequirePackage{fix-cm}
\documentclass{article}
\usepackage[margin=2cm]{geometry}
\usepackage{graphicx}

\usepackage{amsfonts, amsthm,amsmath, algorithm2e}

\usepackage{sidecap}
\RequirePackage[OT1]{fontenc}
\RequirePackage{amsthm,amsmath}  
\RequirePackage[numbers]{natbib}
\RequirePackage[colorlinks,citecolor=blue,urlcolor=blue]{hyperref}
\RequirePackage{verbatim, amscd,epsfig,amstext,amssymb}
\usepackage{mathrsfs}

\usepackage{color}

\usepackage{tikz}
\usetikzlibrary{patterns}
\usepackage{caption}
\usepackage{subcaption}

\newcommand{\lightercolor}[3]{
    \colorlet{#3}{#1!#2!white}
}
\newcommand{\darkercolor}[3]{
    \colorlet{#3}{#1!#2!black}
}
\darkercolor{gray}{50}{dgray}
\lightercolor{gray}{50}{lgray}

\numberwithin{equation}{section}
\theoremstyle{plain}

\newtheorem*{proposition*}{Proposition}
\newtheorem{theorem}{Theorem}

\theoremstyle{definition}
\newtheorem{defn}[theorem]{Definition}

\newtheorem{proposition}[theorem]{Proposition}
\newtheorem{lemma}[theorem]{Lemma}
\newtheorem*{thm*}{Theorem}
\newtheorem*{theorem*}{Theorem}

\newtheorem{prop}[theorem]{Proposition}
\newtheorem{corollary}[theorem]{Corollary}
\usepackage{algpseudocode}

\usetikzlibrary{calc}
\tikzset{
l/.style={gray,dashed},
ll/.style={red,dashed},
grn/.style=green!70!black,
b/.style=blue,
r/.style=red,
p/.style=dark gray}
\usepackage{caption}

\usepackage{xcolor}

\usetikzlibrary{plotmarks}

\renewcommand{\L}{\mathcal{L}}

\newcommand{\C}{\operatorname{C}}
\newcommand{\sgn}{\operatorname{sgn}}

\newcommand{\X}{\mathbb{ X}}

\newcommand{\R}{\mathbb{ R}}

\newcommand{\D}{\mathcal{D}}

\newcommand{\maxx}{\text{max}}
\newcommand{\minn}{\text{min}}

\newcommand{\diag}{\Delta}
\newcommand{\Diag}{\Delta}
\usepackage{url}

\begin{document}

\title{Medians of populations of persistence diagrams}

\author{Katharine Turner}

%
%
%
%
%



\maketitle
\begin{abstract}
Persistence diagrams are common objects in the field of Topological Data Analysis. They are topological summaries that capture both topological and geometric structure within data. Recently there has been a surge of interest in developing tools to statistically analyse populations of persistence diagrams, a process hampered by the complicated geometry of the space of persistence diagrams.
In this paper we study the median of a set of diagrams, defined as the minimizer of an appropriate cost function analogous to the sum of distances used for samples of real numbers. We then characterize the local minima of this cost function  and in doing so characterize the median. We also do some comparative analysis of the properties of the median and the mean.
\end{abstract}

%
%
%
%





\section{Introduction} 

%

Topological data analysis (TDA) is a rapidly growing field that uses multi-scale topological features to find and analyse structure in data. An important use of TDA is as a preprocessing tool, providing topological summaries that may be more tractable than the raw information, and perhaps highlight geometric and topological features that are of particular interest. Examples of applications include the analysis of the shape of human jaws \cite{gamble2010exploring}, plant root systems \cite{bendich2010computing}, shapes of calcanei bones of various primates \cite{turner2014persistent} and retrieval of trademark symbols \cite{cerri2006retrieval}. 

Persistence diagrams are a common topological summary statistic. Each diagram is a discrete summary of how the homology evolves over a nested sequence of topological spaces parameterized by some tuning parameter. Homology computes features such as the number of connected components, loops, tunnels, and void. Although persistent homology is defined via homology, the  tuning parameter captures geometric information. 



As part of the growing field of object oriented data analysis we wish to consider the structure of the space of persistence diagrams when performing statistical analysis. This is complicated by the geometry of the space of persistence diagrams - it is infinite in dimension and with no upper bound on curvature. At no point does it locally look Euclidean. The papers \cite{mileyko2011probability, turner2014frechet} study the geometric and analytic properties of the space of persistence diagrams equipped with the metric analogous to the $2$-Wasserstein distances of probability measures and to the $L^2$ distances of functions on a discrete set. We denote this space $(\D, d_2)$. In \cite{mileyko2011probability} it was shown that it is possible to define the mean and variance of ``nice'' distributions via the Fr\'{e}chet function. In \cite{turner2014frechet} it was shown that $(\D, d_2)$ was a non-negatively curved Alexandrov space and used this structure to characterize the mean of a population of persistence diagrams and to provide an algorithm to compute it. 

After the mean, the median is next most common statistic to describe the center of a distribution. Furthermore, the median is often more robust than the mean. The purpose of this paper is to provide an analogous analysis for the space of the persistence diagrams under the corresponding metric for $p=1$ (instead of $p=2$) and to characterize the median of a population of persistence diagrams. We also are interested in comparing the properties of the mean and the median of populations of persistence diagrams.

In section $2$ we establish geometric properties of the space of persistence diagrams, such as curvature and connected components, for a natural family of metrics $\{(\D, d_p)\}$ that are analogous to the $p$-Wasserstein distances on the space of probability measures and to the $L^p$ distances on the space of functions on a discrete set. Just as spaces of functions equipped with an $L^2$ metric has nicer properties to those with an $L^p$ metric with $p\neq 2$, we see that although $(\D, d_2)$ is a non-negatively curved Alexandrov space, this does not hold for $(\D, d_p)$ when $p\neq 2$.

The median of a population is the location that minimizes the average distances to each of the members of the population. In section $3$ we study the median of a population of persistence diagrams. Unfortunately the proofs in \cite{turner2014frechet} for characterizing and computing the mean require the extra Alexandrov space structure available in the case when $p=2$ (rather than $p=1$ which is the scenario for the median). This implies that we needed to develop new methods in order to characterize the median of a population of persistence diagrams. In the appendix we summarize how these new methods provide a new proof for previous results about the mean.  Another aspect where the analysis here differs from that in \cite{turner2014frechet} is that we consider persistence diagrams containing points with infinite persistence. 

We can define the mean as the location that minimizes the average distance to the sample set. Unlike the mean, the median of a population of even size is not unique.  For $N$ odd, the median of a set of real numbers $a_1, a_2, \ldots, a_N$, written in non-decreasing order, is $a_{(N+1)/{2}}$. If $N$ is even, then every number in the interval $[a_{N/{2}}, a_{(N+2)/{2}}]$ minimizes the average distance to the sample set and hence would be a valid choice as a median. To overcome this lack of uniqueness the general convention is to declare the median to be the midpoint of $[a_{{N}/{2}}, a_{(N+2)/{2}}]$. Analogous conventions could be applied for medians of an odd number of persistence diagrams. However, the statements of theory and the clarity of exposition quickly becomes less clear. For this reason we will be restricting our attention to the case when the size of the populations is odd. 

In section $4$ we compare various properties of the median to the mean. We bound the number of off diagonal points in the median compared to the mean, show the median is more robust than the mean, and prove that the mean is generically unique while the median is not.

%
%


%

\subsection{Related work}
Statistical analysis of persistence diagrams is an example of object oriented data analysis. This is an emerging area of statistics  where the goal is to  develop and apply statistical methods to objects such as functions, images, graphs or trees (e.g. \cite{lu2014object,marron2015functional, marron2014overview, ramsay2006functional,sangalli2014object,wang2007object}) while respecting the structure of these objects. 

There has been a growing movement of developing statistical methods to understand populations of persistence diagrams.  This includes a significant body of work (for example see \cite{bubenik2010statistical, blumberg2014robust, chazal2013bootstrap, chazal2013optimal, fasy2014confidence}) which studies statistical methodology using the \emph{bottleneck distance}, which is effectively the $L^\infty$ distance within the space of persistence diagrams. 
%
%

There has also been a parallel movement of statistical analysis in TDA by converting persistence diagrams into other functional summaries. These new topological summaries lie in spaces where it is easier to modify traditional statistical methods such as computing means or performing t-tests. However, this can be at the expense of making it harder to provide topological interpretations of the results. Examples of such functional summaries includes persistence landscapes \cite{bubenik2013persistence} and persistent homology rank functions \cite{robins2015principal}.

Other related work investigates the homology or persistent homology of random topological objects (for example \cite{adler2014crackle, bobrowski2011distance, kahle2011random, kahle2013limit, yogeshwaran2014random, yogeshwaran2015topology}) including limit theorems and analysis of simulations.

\section{Geometry of the space of Persistence Diagrams}

\subsection{Background theory on persistent homology and persistence diagrams}

Here we will provide a very brief introduction to persistent homology. A more complete coverage can be found in \cite{edelsbrunner2010computational}. Although homology can be computed over any ring, to compute persistent homology we need a field. This is usually $\mathbb{F}_2$ for computational purposes. 

To define persistent homology, we start with a nested sequence of topological spaces,
\begin{align}\label{filteredspace}
X_0 \subseteq X_1 \subseteq X_2 \subseteq \ldots X_n =X.
\end{align}
Often this sequence arises from the sublevel sets of a function, $f \colon X \to \R$, with $X_i =
f^{-1}(-\infty,t_i]$ for a sequence $-\infty\leq t_0 \leq t_1 \leq \ldots \leq t_n\leq \infty$. 
The sequence induces linear maps on homology for any dimension r:
\[H_r(X_0) \to H_r(X_1) \to \ldots \to H_r(X_n).\]

We are interested in when homology classes appear and disappear in this sequence.
Let $\phi^j_i \colon H_r(X_i) \to H_r(X_j)$ be the linear map on homology induced by the inclusion $X_i\to X_j$. Observe that if $i<j<k$ then $\phi_j^k \circ \phi_i^j = \phi_i^k$. The homology
class $\gamma \in H_r(X_i)$ is said to be born at $X_i$ if it is not in the image of $\phi_{i-1}^i$. This same class is said to die
at $X_j$ if its image in $H_r(X_{j-1})$ is not in the image of $\phi_{i-1}^{j-1}$, but its image in $H_r(X_j)$ is in the image of $\phi_{i-1}^j$. In the case that the spaces arose from the level sets of a function $f$ as defined above, we define say that $\gamma$ is born at time $t_i$, dies at time $t_j$, and its lifetime is $t_j-t_i$.

A discrete description of the persistent homology of \eqref{filteredspace} is the multiset of points in the extended plane where we include $(t_i,t_j)$ with multiplicity the dimension of the set of homology classes born at $t_i$ and dying at $t_j$. This is the information recorded as the ($r$th dimensional) persistence diagram. 

Although we have here restricted our definition of persistent homology to that of a finite nested sequence, it can be defined more generally for nested parameterized families of topological spaces $\{X_t\colon t\in \R\}$ with the condition that $X_s\subseteq X_t$ whenever $s\leq t$ (alongside some technical finiteness conditions). We can think of this as a filtered space with filtration parameter $t$. For ``nice'' filtered spaces the persistence diagram can be defined analogously to the finite case. For details see \cite{crawley2015decomposition}.

It is worth observing that although we are computing homology we are often capturing geometric information through the filtration parameter. For example, if we are considering a filtration of $\R^2$ by the distance from a circle of radius $R$ then the $1$st dimensional persistent homology will have exactly one class born at $0$ (the time the circle first appears) and dying at $R$ (the time when the circle is filled in).

Before giving a technical definition of a persistence diagram we will need to introduce some notation. Let $\R^{2+}=\{(a,b)\in \R^2: a<b\}$. This is the part of the plane above the diagonal. Let $\Diag$ denote an abstract element representing the diagonal $\{(x,x)\colon x\in \R\}$.

%
%

Since we wish to also want to consider persistent homology classes of infinite duration we will also want to include lines at infinity. Let $\L_{-\infty}:=\{(-\infty, b)\colon b\in \R\}$ and $\L_{\infty}:=\{(a,\infty)\colon a\in \R\}$. 
Persistence diagrams will be multiset of points in $\L_\infty \cup \L_{-\infty}\cup \R^{2+}\cup \Delta$. For tractability we will impose some finiteness conditions, namely that only finitely many classes have infinite lifetimes and that the sum of all the finite lifetimes is finite. This restriction is not onerous in applications where generally we have finite sized data as input.

\begin{defn}
A \emph{persistence diagram} $X$ is a multiset of $\L_\infty \cup \L_{-\infty}\cup \R^{2+}\cup \Diag$ such that 
\begin{itemize}
\item The number of elements in $X|_{\L_{\infty}}$ and $X|_{\L_{-\infty}}$ are finite
\item $\sum_{(x_i,y_i)\in X|_{\R^2+}}(y_i-x_i)<\infty$ 
\item There are countably infinite copies of $\Diag$
\end{itemize}
\end{defn}

%

%
\subsection{Metrics on the space of persistence diagrams}\label{sec:metric}

In this paper we differ slightly from the historical definition of a persistence diagram. Instead of including every point along the diagonal of $\R^{2}$ with infinite multiplicity we include countably many copies of the diagonal. This alternative description still has the same distances between different persistence diagrams but simplifies statements, arguments and calculations. This is because we do not need to specify which point on the diagonal is being used.

Let $\D$ denote the space of all persistence diagrams. We will consider a family of metrics which are analogous to the $p$-Wasserstein distances on the space of probability measures and to the $L^p$ distances on the space of functions on a discrete set. $\R^{2+}$ inherits natural $L^p$ distances from $\R^2$. For $p\in [1,\infty)$ we have 
$\|(a_1, b_1) - (a_2, b_2)\|_p^p = |a_1-a_2|^p + |b_1-b_2|^p$ and 
$\|(a_1, b_1) -(a_2, b_2)\|_\infty = \max \{ |a_1-a_2|, |b_1, b_2|\}$.

Recall that $\Delta$ represents the diagonal in $\R^2$. With a slight abuse of notation we write $\|(a, b) - \Delta\|_p $ to denote the shortest $L^p$ distance from $(a,b)$ in to a point in the diagonal set in $\R^2$.  Thus 
\[\|(a, b) - \Delta\|_p = \inf_{t\in \R}\|(a,b)-(t,t)\|_p= 2^{\frac{1}{p}-1}|b-a|\] for $p<\infty$, and $\|(a, b)-\Delta\|_\infty=\inf_{t\in \R}\|(a,b)-(t,t)\|_\infty = |y-x|/2$. Both $ \L_{-\infty}$ and $\L_\infty$ inherit natural $L^p$ distances from the $L^p$ metric on $\R$; $\|(-\infty, b_1)-(-\infty, b_2)\|_p=|b_1-b_2|$ and $\|(a_1,\infty)-(a,\infty)\|_p=|a_1-a_2|$. We should also think of $\L_{-\infty}$, $\L_{-\infty}$ and $\R^{2+}\cup \Delta$ as three separate disjoint parts of a larger space.

Given  persistence diagrams $X$ and $Y$ we can consider all the bijections from the set of off diagonal points and copies of $\Delta$ in $X$, to the set of off diagonal points and copies of $\Delta$ in $Y$. This set is non-empty as it contains the bijection which matches everything to a copy of $\Delta$ in the other diagram. Each bijection provides a transport plan from $X$ to $Y$. Analogous to the definition of Wasserstein distances, we will define our family of metrics in terms of the cost of most efficient transport plan. 

For each $p \in [1, \infty)$ define \[d_p(X,Y) = \left( \inf_{\text{bijections }\phi:X \to Y } \sum_{x\in X} \|x-\phi(x)\|_p^p\right)^{1/p}.\]
 and $d_\infty(X,Y) = \inf_{\text{bijections }\phi:X \to Y} \sup_{x\in X} \|x-\phi(x)\|_\infty.$

These distances may be infinite - for example if $X$ and $Y$ contain a different number of points in $\mathcal{L}_\infty$ then $d_p(X,Y)=\infty$ for all $p$.

We will call a bijection between diagrams \emph{optimal for $d_p$} if it achieves the infimum in the definition of $d_p$ and this distance is finite. Figure \ref{fig:egbijection} illustrates an example of an optimal bijection.

\begin{figure}

\begin{center}
\begin{tikzpicture}[scale=.35]
\draw []   (-1,-1) -- (10,10);
\draw [->] (-1, 0) -- (10,0);
\draw [->] ( 0,-1) -- (0,10);

\draw[dashed] (2,4) -- (3,3);
\draw[dashed] (0,4) -- (1,5); 
\draw[dashed] (3,8) -- (2,8);
\draw[dashed] (6,7) -- (6.5,6.5);
\draw[dashed] (8,9) -- (8.5,8.5);

\draw plot[only marks, mark=square*,mark options={ mark size=4pt, fill=white}] coordinates{
(0,4) (2,4)(2,8)};

\draw plot[only marks, mark=triangle*,mark options={mark size=6pt, fill=white}] coordinates{
(1,5)(3,8) (6,7)(8,9)};

\end{tikzpicture}
\end{center}
\caption{The dashed lines indicate an optimal bijection from the persistence diagram containing the square points (and copies of $\Delta$) to the persistence diagram containing the triangle points (and copies of $\Delta$).}\label{fig:egbijection}
\end{figure}
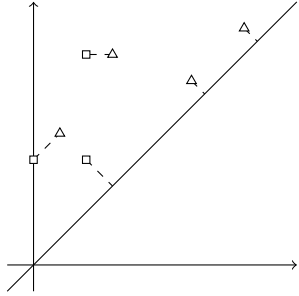

Given the same pair of diagrams but different values of $p$, different bijections may be optimal. Furthermore, optimal bijections for a given $p$ are not necessarily unique. For example, let $X$ and $Y$ to be diagrams containing pairs of opposite corners of a square located far from the diagonal. Because of symmetry, bijection the points vertically or horizontally involves the same cost. This example works for every $p\in [1,\infty]$. Other examples  of non-uniqueness can involve $\Delta$; it may be equally efficient to match two points to each other as it is to match both with a copy of $\Delta$.

In theory, for every pair $p,q\in [1,\infty]$ one could construct a distance function of the form 
\[ \inf_{\phi:X \to Y} \left(\sum_{x\in X} \|x-\phi(x)\|_q^p\right)^{1/p}\] with $p$ and $q$ different.  Some of the computational topology literature uses a family of metrics $d_{W_p}$ where $p$ varies but $q=\infty$ is fixed. 
The families $\{d_p\}$ and $\{d_{W_p}\}$ share many properties. The metrics $d_p$ and $d_{W_p}$ are bi-Lipschitz equivalent as
for any $x,y \in \R^2$ we have $\|x-y \|_\infty \leq \|x -y \|_p \leq 2 \|x - y \|_\infty$,
 implying $d_{W_p}(X,Y)\leq d_{p}(X,Y) \leq 2 d_{W_p}(X,Y)$. 
Any stability results for $\{d_p\}$ or $\{d_{W_p}\}$ would extend (with minor changes in constant) to stability results for the other.

We feel that the choice of setting $q=p$ is cleaner in theory and in practice. The coordinates of the points within a persistence diagram have particular meanings; one is the birth time and one is the death time. They are often infinitesimally independent (even though not globally so). For example, if we have generated our persistence diagram from the distance function to a point cloud then each persistence class has its birth and death time (infinitesimally) determined by the location of two pairs of points which are often distinct.  Whenever these pairs are distinct, moving any of these four points will change either the birth or the death but not both. The distinctness of the treatment of birth and death times as separate qualities may seem more philosophically pleasing to the reader in the setting of barcodes.

Another argument in favor of using $p=q$ is computational power. 
The mean and the median are defined to be the minimizers of cost functions involving $d_2$ and $d_1$ respectively. Computationally the mean and median have far nicer characterizations, with easy algorithms to find them, when they are defined using the metrics where $p=q$. This is discussed later in section \ref{sec:defmed}.

\subsection{Geometry of the space of persistence diagrams}

In this section we will describe the connected components of $\D$ and show $\D$ is a geodesic space, but first we need to introduce some notation. For each $p<\infty$ we can define the distance between finite multisets $A,B$ of points in $\L_\infty$ as
\[d^{\L_\infty}_p(A, B)^p =  \inf_{\text{bijections } \phi\colon A \to B}\sum_{a\in A} \|a-\phi(a)\|_p^p\] when $|A|=|B|$ and infinity otherwise. 
We also set \[d^{\L_\infty}_\infty(A, B) =  \inf_{\text{bijections } \phi\colon A \to B}\sup_{a\in A} \|a-\phi(a)\|_\infty\] whenever $|A|=|B|$ and infinity otherwise.  We define $d^{\L_{-\infty}}_p$ similarly.
Let $\D^{(k,l)}$ denote the space of persistence diagrams containing exactly $k$ points in $\L_{-\infty}$, and exactly $l$ points in $\L_{\infty}$.

\begin{lemma}\label{lem:splitdist}
Let $p\in [1,\infty]$. The connected components of $(\D, d_p)$ are $\D^{(k,l)}$, with $k, l$ non-negative integers.
Let $X,Y\in \D^{(k,l)}$. For $p<\infty$,
\begin{align*}
d_p(X,Y)^p= d_p(X|_{\R^{2+}\cup\Delta}, Y|_{\R^{2+}\cup\Delta})^p + d^{\L_{\infty}}_p( X|_{\L_\infty} ,Y|_{\L_\infty})^p+ d^{\L_{-\infty}}_p( X|_{\L_{-\infty}} ,Y|_{\L_{-\infty}})^p 
\end{align*}
and 
\[d_\infty(X,Y)=\max\{d_\infty(X|_{\R^{2+}\cup\Delta}, Y|_{\R^{2+}\cup\Delta}), d^{\L_{\infty}}_\infty( X|_{\L_\infty} ,Y|_{\L_\infty}), d^{\L_{-\infty}}_\infty( X|_{\L_{-\infty}} ,Y|_{\L_{-\infty}})\}.\] 

\end{lemma}

The proof of this lemma follows from the observations that the connected components of $\L_\infty \cup \L_{-\infty}\cup \R^{2+}\cup\Delta$ are $\L_\infty$, $\L_{-\infty}$ and $\R^{2+}\cup\Delta$, and that $d_p( X|_{\R^{2+}\cup\Delta}, Y|_{\R^{2+}\cup\Delta})$ is always finite.

Let $(\X,d)$ be a metric space. A curve $\lambda\colon [0,1]\to \X$ is called a \emph{geodesic} if there exists constant $C>0$ such that $d(\lambda(t_1), \lambda(t_2))=C|t_1-t_2|$ for all $t_1, t_2 \in [0,1]$ with $|t_1-t_2|$ sufficiently small. $(\X,d)$ is called a \emph{geodesic space} if for every $x,y\in \X$ there exists a geodesic $\lambda:[0,1] \to \X$ such that $\lambda(0)=x$ and $\lambda(1)=y$.

\begin{proposition}
$(\D,d_p)$ is a geodesic space for all $p \in [1,\infty]$.
\end{proposition}
\begin{proof}
Fix $p\in [1,\infty)$ and $X,Y \in \D$ with $d_p(X,Y)<\infty$. We want to find a bijection $\phi$ such that $ d_p(X,Y)^p=\sum_{x\in X} \|x-\phi(x)\|_p^p$ and from this bijection to construct a geodesic from $X$ to $Y$.

Let $\{\phi_i\}$ be a sequence of bijections such that 
$\lim_{i\to\infty}\sum_{x\in X} \|x-\phi_i(x)\|_p^p = d_p(X,Y)^p.$
Fix some  off diagonal    point $\hat{x} \in X$. The sequence $\{\phi_i(\hat{x})\}$ must have a convergent subsequence $\{\phi_{i_j}(\hat{x})\}$ which converges either to an off diagonal  point or to $\diag$ . This limit point must be in $Y$ and we set $\phi(\hat{x})$ to be this limit point.  Observe our subsequence also satisfies
$\lim_{j\to\infty}\sum_{x\in X} \|x-\phi_{i_j}(x)\|_p^p = d_p(X,Y)^p.$ 

We now replace our original sequence of bijections $\{\phi_i\}$ with the subsequence $\{\phi_{i_j}\}$. In this manner we can determine a choice $\phi(x)$ for each  off diagonal    point $x\in X$. Similarly we can determine $\phi^{-1}(y)$ for all the  off diagonal    points $y \in Y$. Since we are always considering subsequences of previous subsequences we have consistency in our choices.

Since there are only countably many points off the diagonal in the diagrams $X$ and $Y$ combined we can find an optimal bijection $\phi:X\to Y$ 
Let $X_t$ be the diagram with off diagonal points $\{(1-t)x+t\phi(x)\colon x\in X\}$ and set define the path $\lambda:[0,1] \to \D$ by $\lambda(t)=X_t$. By observation $X_0=X$, $X_1=Y$, and $\lambda$ is a geodesic.

The $p=\infty$ case is similar. Let  $\{\phi_i\}$  be a sequence of bijections such that $$\lim_{i\to\infty} \maxx_{x\in X} \|x-\phi_i(x)\|_\infty = d_\infty(X,Y)$$ and proceed as in the case $p \in [1,\infty)$ to produce a bijection $\phi$, by assigning the values of $\phi(x)$ and restricting to appropriate subsequences, such that $d_\infty(X,Y)=\maxx_{x\in X} \|x-\phi(x)\|_\infty$. This bijection will determine a geodesic by the same reasoning in the $p\in [1,\infty)$ case.
\end{proof}

\subsection{Curvature bounds on the space of persistence diagrams}

In order to understand the space of persistence diagrams it is useful to analyze its curvature. Alexandrov spaces are geodesic spaces with curvature bounds. They come in two different forms;  either their curvature is bounded from above (also known as $CAT$ spaces) or their curvature is bounded from below. A bound on curvature in geodesic space $(\X,d)$ is defined using comparison triangles. For each $\kappa\in \R$ there is a model space $M_\kappa$ with constant curvature $\kappa$. We compare triangles in $\X$ to triangles in $M_\kappa$. Take three points $x,y,z$. If $\kappa>0$ we require $d(x,y) + d(y,z) + d(z,x) \leq \sqrt{2\pi/\kappa}.$ These define a triangle $\Delta(x,y,z)$ in $\X$. We can build a comparison triangle $\Delta(\tilde{x},\tilde{y},\tilde{z})$ in the model space $M_\kappa$  whose sides have the same length as the sides of $\Delta(x,y,z)$. The curvature of $\X$ is bounded from below (above) by $\kappa$ if, for every triangle $\Delta(x,y,z)$ in $\X$, the distances between the points in $\Delta(x,y,z)$ are less than or equal (respectively greater than or equal) the corresponding points in the comparison triangle $\Delta(x',y',z')$ in $M_\kappa$. For more details see \cite{Kirk}.

A $CAT(k)$ space is a geodesic space whose curvature is bounded from above by $k$. $CAT$-spaces, in particular $CAT(0)$ spaces, have desirable properties. For example, the barycenter of any measure in a $CAT(0)$ space is unique and in a $CAT(k)$ space there is a  length $D_k$ such that 
balls of radius $D_k$ are contractible. We first confirm that $(\D,d_p)$ is not a $CAT$-space.

\begin{proposition}\label{notCAT}
For all $k>0$ and $p\in [1,\infty]$, $(\D, d_p)$ is not in $\mbox{CAT}(k)$.
\end{proposition}
\begin{proof}
If $(\D, d_p)$  is a $\mbox{CAT}(k)$ space then there is a constant $K>0$ such that for all pairs $X,Y\in (\D, d_p)$ with
$d_p(X,Y)^2<K$ there is a unique geodesic between them \cite{Kirk}. However, we can find pairs of diagrams
$X$ and $Y$ which are arbitrarily close such that there are two distinct geodesics between them. One example is by taking $X$ to be a diagram with two diagonally opposite corners of a square set far from the diagonal and $Y$ the diagram with the other two corners. This is illustrated in Figure \ref{fig:notcat0}. The horizontal and vertical paths are equally optimal and we may choose the square to be as small as we wish.
\end{proof}

\begin{figure}[hbt]
\begin{minipage}{0.45\linewidth}\centering

\begin{tikzpicture}[scale=.4]
\draw []   (-1,-1) -- (10,10);
\draw [->] (-1, 0) -- (10,0);
\draw [->] ( 0,-1) -- (0,10);

\draw[dashed] (1,6) -- (1,8);
\draw[dashed] (3,6) -- (3,8);

\draw plot[only marks, mark=square*,mark options={ mark size=4pt, fill=white}] coordinates{
(1,6) (3,8)};

\draw plot[only marks, mark=triangle*,mark options={mark size=6pt, fill=white}] coordinates{
(1,8)(3,6)};

\end{tikzpicture}
\end{minipage}
\begin{minipage}{0.45\linewidth}\centering
\begin{tikzpicture}[scale=.4]
\draw []   (-1,-1) -- (10,10);
\draw [->] (-1, 0) -- (10,0);
\draw [->] ( 0,-1) -- (0,10);

\draw[dashed] (1,6) -- (3,6);
\draw[dashed] (1,8) -- (3,8);

\draw plot[only marks, mark=square*,mark options={ mark size=4pt, fill=white}] coordinates{
(1,6) (3,8)};

\draw plot[only marks, mark=triangle*,mark options={mark size=6pt, fill=white}] coordinates{
(1,8)(3,6)};

\end{tikzpicture}
\end{minipage}
\caption{Two different optimal bijections between the triangle and the square diagrams.}\label{fig:notcat0}
\end{figure}
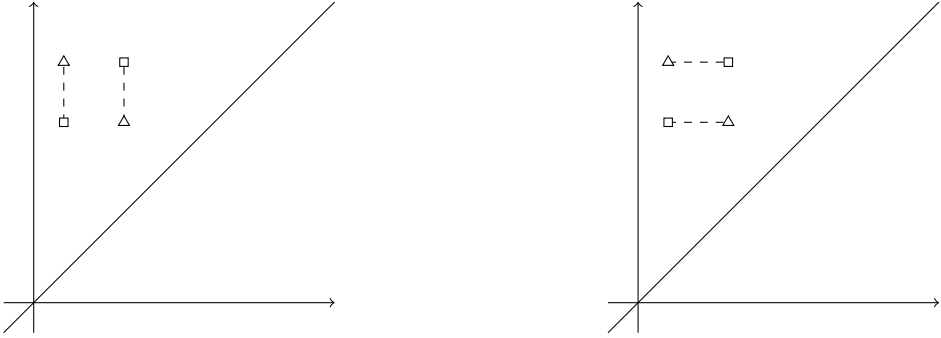

Alexandrov spaces have nice properties. For example, for Alexandrov spaces we can define tangent cones (analogous to tangent planes), and exponential maps. From \cite{Ohta}  we know that a geodesic space $(\X,d)$ is an Alexandrov space with 
curvature bounded from below by zero if, and only if, for every geodesic $\gamma:[0,1] \to \X$ from $X$
to $Y$, and every $Z \in \X$ we have 
\begin{align}\label{eq:alex}
d(Z,\gamma(t))^2 \geq  t d(Z,Y)^2 + (1-t) d(Z,X)^2  -t(1-t) d(X,Y)^2.
\end{align}
Using this characterization \cite{turner2014frechet} showed that $(\D, d_2)$ is an Alexandrov space with curvature bounded below by zero.

In contrast,  using different counterexamples for $p \in [1,2)$ and for $p \in (2, \infty]$ we can show that this curvature bound does not hold when $p\neq 2$.  These counterexamples are illustrated in Figure \ref{fig:counterex}.

 Let $p \in [1,2)$ and $t=1/2$. Let $X, Y$ and $Z$ be a persistence diagram with only one off diagonal point each in them at $x=(1,4), y=(1,6)$ and $z=(0,5)$ respectively. The midway point between $X$ and $Y$ (playing the role of $\gamma(1/2)$) is the diagram with the point $w=(1,5)$. This set up is shown in Figure \ref{fig:counter2inf}.
 We can calculate $d_p(Z,\gamma(1/2))^p= \|z-w\|_p^p =1$, $d_p(Z,X)^p=\|z-x\|_p^p =2$, $d_p(Z,Y)^p =\|z-y\|_p^p =2$ and $d_p(X,Y)^p=\|x-y\|_p^p=2^p$. Together they imply
 \[ \frac{1}{2} d_p(Z,Y)^2 + \frac{1}{2} d_p(Z,X)^2  -\frac{1}{4} d_p(X,Y)^2 =2^{2/p} -1.\]
 But $2^{2/p} -1>1= d_p(Z,\gamma(1/2))^2$ when $1\leq p<2$. This contradicts equation \eqref{eq:alex} and hence $(\D,d_p)$ is not an Alexandrov space with curvature bounded below by zero.

 Now let $p \in (2, \infty)$ and $t=1/2$. Let $X, Y$ and $Z$ be a persistence diagram with only one off diagonal point each in them at $x=(0,4), y=(2,6)$ and $z=(0,6)$ respectively. The midway point between $X$ and $Y$ (playing the role of $\gamma(1/2)$) is the diagram with the point $w=(1,5)$. This set up is shown in Figure \ref{fig:counter12}. Here
 $d_p(Z,\gamma(1/2))^p= \|z-w\|_p^p =2$,
 $d_p(Z,X)^p=\|z-x\|_p^p=2^p$,
 $d_p(Z,Y)^p =\|z-y\|_p^p =2^p$, and
 $d_p(X,Y)^p=\|x-y\|_p^p = 2^{p+1}$.

 %
 Together they imply
 \begin{align*}
  \frac{1}{2} d_p(Z,Y)^2 + \frac{1}{2} d_p(Z,X)^2  -\frac{1}{4} d_p(X,Y)^2 =2^2 -2^{2/p}
  \end{align*}
  But $ 2^2 -2^{2/p}>2^{2/p}= d_p(Z,\gamma(1/2))^2$
  when $p>2$. This contradicts equation \eqref{eq:alex} and hence $(\D,d_p)$ is not an Alexandrov space with curvature bounded below by zero.
  
  With this same set up we can calculate
  $d_\infty(Z,\gamma(1/2))= \|z-w\|_\infty = 1$,
  $d_\infty(Z,X)=\|z-x\|_\infty = 2$,
  $d_\infty(Z,Y) =\|z-y\|_\infty = 2$, and
  $d_\infty(X,Y)=\|x-y\|_\infty=2$. Hence
  %
  \begin{align*}
  \frac{1}{2} d_\infty(Z,Y)^2 + \frac{1}{2} d_\infty(Z,X)^2  -\frac{1}{4} d_\infty(X,Y)^2 &=3>1= d_\infty(Z,\gamma(1/2))^2.
  \end{align*}
  This contradicts equation \eqref{eq:alex} and hence $(\D,d_\infty)$ is not an Alexandrov space with curvature bounded below by zero.

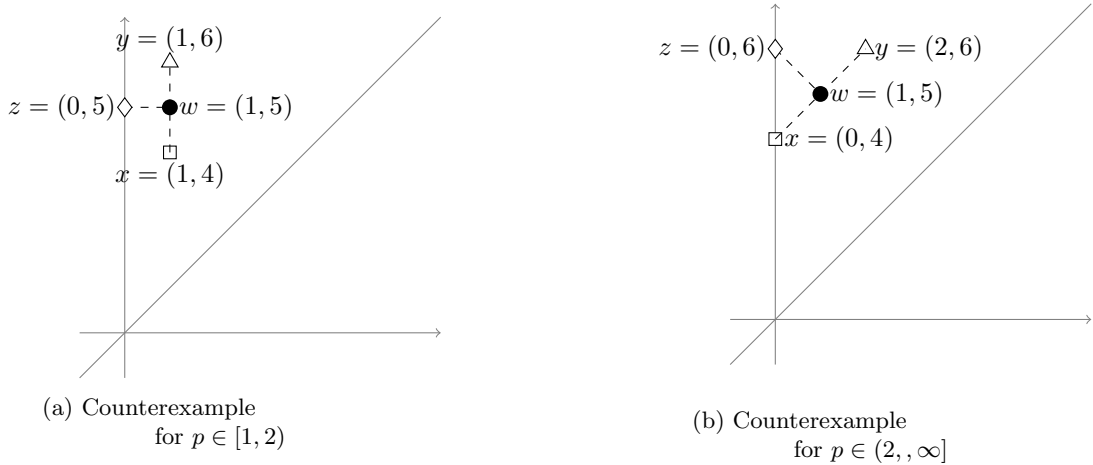
\begin{figure}[h]
\begin{minipage}{0.485\linewidth}\centering
\begin{center}
\begin{tikzpicture}[scale=.6]
\draw[gray, ->] (0,-1)--(0,7);
\draw[gray, ->](-1,0)--(7,0);
\draw[gray](-1,-1)--(7,7);

\draw plot[mark=square*,mark options={mark size=4pt, fill=white}] coordinates{
(1,4)} node [below]  {$x=(1,4)$};
\draw[dashed](1,4)--(1,5) ;
\draw[dashed](1,5)--(1,6);
\draw[dashed](0,5)--(1,5);

\draw plot[mark=triangle*,mark options={mark size=6pt, fill=white}] coordinates{
(1,6)} node [above]  {$y=(1,6)$};

\draw plot[mark=diamond*,mark options={mark size=6pt, fill=white}] coordinates{
(0,5)} node [left]  {$z=(0,5)$};

\fill[] (1, 5) circle (5pt) node[right] {$w=(1,5)$};

\end{tikzpicture}
\subcaption{Counterexample \newline for $p \in [1,2)$}\label{fig:counter2inf} 
\end{center}
\end{minipage}
\begin{minipage}{0.485\linewidth}\centering
\begin{center}
\begin{tikzpicture}[scale=.6]
\draw[gray, ->] (0,-1)--(0,7);
\draw[gray, ->](-1,0)--(7,0);
\draw[gray](-1,-1)--(7,7);

\draw plot[mark=triangle*,mark options={mark size=6pt, fill=white}]coordinates{
(2,6)} node [right]  {$y=(2,6)$};
\draw plot[mark=square*,mark options={mark size=4pt, fill=white}]coordinates{
(0,4)} node [right]  {$x=(0,4)$};

\draw[dashed](0,4)--(1,5) ;
\draw[dashed](1,5)--(2,6);
\draw[dashed](0,6)--(1,5);

\draw plot[mark=diamond*,mark options={mark size=6pt, fill=white}] coordinates{
(0,6)} node [left]  {$z=(0,6)$};

\fill (1, 5) circle (5pt)node[right]{$w=(1,5)$};

\end{tikzpicture}
\end{center}
\subcaption{Counterexample \newline for $p\in (2,,\infty]$}\label{fig:counter12}
\end{minipage}
\caption{In both (a) and (b) the geodesic $\gamma$ from $X=\{x\}$  to $Y=\{y\}$ has midpoint $\gamma(1/2)=\{w\}$. We consider the distances to $Z=\{z\}$. Here we are describing each persistence diagram by its list of off diagonal points. In both examples
$d_p(Z,\gamma(1/2))^2 <  \frac{1}{2} d_p(Z,Y)^2 + \frac{1}{2} d_p(Z,X)^2  -\frac{1}{4} d_p(X,Y)^2$ contradicting \eqref{eq:alex}. Thus $(\D,d_p)$ is not an Alexandrov space with curvature bounded below by zero when $p\in [1,2)\cup(2,\infty]$.\label{fig:counterex}}
\end{figure}

%

%

\section{The median of a population of persistence diagrams}\label{sec:defmed}

Measures of central tendency (such as the mean and the median), or their corresponding measures of variability or dispersion (the variance and the average cost respectively) are common statistics used to describe distributions. 
Central tendencies are solutions for optimizing different cost functions which are based on $p$-Wasserstein metrics. The median corresponds to the case where $p=1$.

%
%
%
%
%
%
The median of a set of real numbers $a_1, a_2, \ldots, a_N$, written in non-decreasing order, is the number $m$ which minimizes the mean absolute deviation function
$F_1^{\R}(x) = \frac{1}{N}\sum_{i=1}^N |a_i-x|.$
The average cost of moving a point in the sample data to the median is $F_1(m)$. For $N$ odd the median is unique and equals $a_{{N+1}/{2}}$.  If $N$ even then every number in the interval $[a_{N/{2}}, a_{(N+2)/{2}}]$ will minimize $F_1^{\R}$ and would be a valid choice as a median. To overcome this lack of uniqueness the general convention is to declare the median to be the midpoint of $[a_{{N}/{2}}, a_{(N+2)/{2}}]$.

More generally the median of a population $\{x_1, x_2, \ldots, x_N\}$ within a connected metric space $\X$ is the minimizer of the function $F_1(y)= \frac{1}{N}\sum_{i=1}^N d(x_i, y)$ where $d$ is an appropriate metric on $\X$.


%
%
%
To define the median of a population of persistence diagrams we need to fix a metric on the space of persistence diagrams. Unlike on the real line there are multiple reasonable options as explored in section \ref{sec:metric}. In this paper we argue that a well-motivated metric is $d_1$. This is for two main reasons. Firstly, the coordinates each have separate meanings and are infinitesimally independent. Thus for taking the median the birth should be the median of the relevant births, the deaths the median of the relevant deaths. Secondly, the computations become significantly easier. If we were to use $d_2$ from \ref{sec:metric} then would need to compute, at some stage, the geometric median of a set of points in the plane.  This is a problem as it was shown in \cite{bajaj1988algebraic} that in general there is no exact algorithm to find the geometric median of a set of $k$ points in the plane.

The \emph{median} is the persistence diagram $m$ which minimizes the cost function
$F_1(Y) = \frac{1}{N}\sum_{i=1}^N d_1(X_i, Y)$. 
The \emph{total cost} is $NF_1(m)$ and the  \emph{average cost} is $F_1(m)$. 

It is worth observing that the median is only defined for populations $\{X_1, \ldots, X_N\}$ that lie within the same connected component of $\D$ as otherwise for every $Y$ there is some $X_i$ such that $d_1(X_i,Y)=\infty$.

\begin{prop}
Fix $k,l$ non-negative integers. Let $X=\{X_1, X_2, \ldots, X_N\}$ be a population of persistence diagrams in $\D^{(k,l)}$. For each $i$ let $X_i^{\R^{2+}}$ denote  ${X_i}|_{\R^{2+}\cup \Delta}$, $X_i^\infty$ denote ${X_i}|_{\L_\infty}$ and  $X_i^{-\infty}$ denote ${X_i}|_{\L_-\infty}$. 

The diagram $Y$ is a median of $X$ if and only if $Y=Y^{\R^{2+}} \cup Y^\infty \cup Y^{-\infty}$ where $Y^{\R^{2+}}$ is a median of $\{X_1^{\R^{2+}}, X_2^{\R^{2+}}, \ldots, X_N^{\R^{2+}}\}$, $Y^\infty$ is a median of $\{X_1^{\infty}, X_2^{\infty}, \ldots, X_N^{\infty}\}$ and $Y^{-\infty}$ is a median of $\{X_1^{-\infty}, X_2^{-\infty}, \ldots, X_N^{-\infty}\}$
\end{prop}

\begin{proof}
Let $Z\in \D^{(k,l)}$. From Lemma \ref{lem:splitdist} we know that \[d_1(Z,X_i)=d_1(Z|_{\R^{2+}}, X_i^{\R^{2+}}) + d^{\L_{\infty}}_1( Z|_{\L_\infty} ,X_i^{\infty})+ d^{\L_{-\infty}}_1( Z|_{\L_{-\infty}} ,X_i^{-\infty})\] and hence we can write $F_1$ as a sum of three independent sums
\[F_1(Z)=\frac{1}{N}\sum_{i=1}^N d_1(Z|_{\R^{2+}}, X_i^{\R^{2+}})  + \frac{1}{N}\sum_{i=1}^N d^{\L_{\infty}}_1( Z|_{\L_\infty} ,X_i^{\infty}) + \frac{1}{N}\sum_{i=1}^N d^{\L_{-\infty}}_1( Z|_{\L_{-\infty}} ,X_i^{-\infty}).\]
If $Y$ is a median then it must minimize each of these sums.
\end{proof}

In the remainder of this section we will characterize the median of a population of multisets in $\L_\infty$ or in $\L_{-\infty}$. The next section will address the more complicated characterization of medians of populations in $\D^{(0,0)}$.

\begin{lemma}
Fix $N$ odd. Let $A_1, A_2, \ldots , A_N$ be each multisets of exactly $k$ real numbers. Label the elements of each $A_i$ so that $A_i = \{a_{i,1}, a_{i,2}, \ldots, a_{i,k}\}$ with $a_{i,1} \leq a_{i,2} \leq \ldots \leq a_{i,k}$. Set $B= \{b_1, b_2, \ldots, b_k\}$ where $b_j$  is the median of $\{a_{1,j}, \ldots a_{N,j}\}$. Then $B$ is the unique multiset of $k$ real numbers that minimizes
\[f_1:Y \mapsto \sum_{i=1}^N \left( \inf_{\phi:A_i \to Y, \phi \text{ bijection}} \sum_{a\in A_i} |a-\phi (a)|\right).\]
\end{lemma}
%
%
%
\begin{proof}
The key to this proof is that for $X=\{x_1, x_2, \ldots , x_k\}$ and $Y=\{y_1, y_2, \ldots , y_k\}$ (each written in non-decreasing order) we have \[\inf_{\phi:X \to Y, \text{ $\phi$ bijection}} \sum_{j=1}^m |x_i-\phi(x_i)| = \sum_{j=1}^m |x_j-y_j|.\] We are not claiming that $\phi\colon x_i\mapsto y_i$ is the unique optimal transport from $X$ to $Y$ (this is not always true) but just that it achieves this optimality.

Suppose the multiset $Y =\{y_1, y_2, \ldots, y_k\}$ (written in non-decreasing order) minimizes $f_1$.
The observation above implies 
\[f_1(Y) = \sum_{i=1}^N \sum_{j=1}^k  |a_{i,j} -y_j| = \sum_{j=1}^k\left(\sum_{i=1}^N  |a_{i,j} -y_j| \right).\]

For each $j$ let $b_j$ be the median of $\{a_{i,j}\}$. Since $N$ is odd, 
$\sum_{i=1}^N  |a_{i,j} -b_j| \leq \sum_{i=1}^N  |a_{i,j} -y_j|$
with equality if and only if $b_j=y_j$. Since $B$ minimizes $f_1$ we conclude that $y_j$ is the median of $\{a_{i,j}\}$
\end{proof}

If $N$ is even, then the median is not unique, even if we use the midpoint convention for populations of real values.
For example if $A_1=\{0, 2\}$ and $A_2=\{6, 12\}$, the two medians would be $\{\text{midpt}[0,6]=3,\text{midpt}[2,12]=8\}$ and $\{\text{midpt}[0,12]=6,\text{midpt}[2,6]=4\}$.
%
%
%
%
%
%
%
%

\subsection{The mean and median of multisets of points in the plane and copies of the diagonal}

We are splitting up our analysis into the different regions $\R^{2+}\cup \Delta$, $\L_\infty$ or $\L_{-\infty}$. This is because these are the disconnected components and all geodesics will keep the points in the persistence diagrams within these different regions separate. Thus in this section we will focus on the points in $\R^{2+}\cup \Delta$, that is those corresponding to persistent homology classes with finite lifetimes.

Before considering the problem of populations of multisets in $\R^{2+}\cup \Delta$ we will first investigate the simpler problem of populations of singletons in  $\R^{2+}\cup \Delta$. Here we are using the definition of median as the minimizer of the sum of $d_1$ distances. Given a population $S$ set $f_S:\R^{2+} \to \R$ by $f_S(z)=\frac{1}{N}\sum_{w\in S}||z-w||_1$. Thus $y$ is a median of $S$ implies it is a minimizer of $f_S$. 

Within the proposition the candidate minimiser is the point whose $x-$ and $y-$ coordinates are each the median of a sets of numbers constructed from the $x-$ and $y-$ coordinates of the points in $S$, and from copies of the diagonal. The idea is that whenever a point $(x,y)$ is matched with the diagonal it is effectively matching it with a point $(t,t)$ with $x\leq t\leq y$. For calculating a median of the $x$ coordinates with $|S|$ larger than the number of the copies of the diagonal, a contribution of $t\in [x,y]$ and a contribution of $\infty$ will have the same effect. Similarly, for calculating the median of the $y$-coordinates, a contribution of $t\in [x,y]$ will have the same effect as that of $-\infty$. This means that from the purposes of calculation, we can use $\pm \infty$ in our lists of coordinates in the proposition below.

\begin{proposition}\label{prop:median}

Fix $N$ odd and suppose $k>N/2$. Let $S=\{(a_1, b_1), (a_2, b_2), \ldots, (a_k, b_k)\}$ and $N-k$ copies of $\diag$ (where the $(a_i,b_i)\in  \R^{2+}$). 
Define $f=f_S|_{\R^{2+}}$, that is
\[f((x,y)) = \sum_{i=1}^k \|(x,y)-(a_i,b_i)\|_1+ \sum_{i=k+1}^N\|(x,y)-\diag\|_1.\]
Let $(\tilde{x},\tilde{y})$ be the point in $\R^2$  where $\tilde{x}$ is the median of $\{a_1, a_2 \ldots a_k\}$ with $N-k$ copies of $\infty$ and $\tilde{y}$ is the median of $\{b_1, b_2, \ldots , b_k\}$ with $N-k$ copies of $-\infty$. If $\tilde{x}< \tilde{y}$ then $(\tilde{x},\tilde{y})$ is the point in $ \R^{2+}$  which minimizes
$f$. 
If $\tilde{x}\geq \tilde{y}$ then $f((x,y))> \sum_{i=1}^k \|\diag-(a_i,b_i)\|_1$ for all $(x,y)\in \R^{2+}$.
\end{proposition}
\begin{proof}
Since $k>N/2$ we know that $\tilde{x}$ and $\tilde{y}$ are finite. Suppose that $\tilde{x}< \tilde{y}$. We want to show $(\tilde{x},\tilde{y})$ is the minimum of $f$. Since $f$ is a convex function over $\R^{2+}$ it is sufficient to show $(\tilde{x}, \tilde{y})$ is a local minimum. 

Consider pairs $(u,v)$ such that 
$|u| <\minn_{a_i\neq \tilde{x}}|\tilde{x}-a_i|$, $|v|<\minn_{b_i\neq \tilde{y}}|\tilde{y}-b_i|,$ and  $|u|+|v|<
\|(\tilde{x},\tilde{y})-\diag\|_1$. This is true for sufficiently small $u$ and $v$. For such $(u,v)$ we have
\begin{align*}
\sum_{i=1}^k &\|(\tilde{x}+u,\tilde{y} +v)-(a_i,b_i)\|_1-\sum_{i=1}^k \|(\tilde{x},\tilde{y})-(a_i,b_i)\|_1 \\
&=|\{i:a_i<\tilde{ x}\}| u + |\{i:a_i>\tilde{ x}\}| (-u) + |\{i:a_i=\tilde{x}\}| |u|\\
&\quad+ |\{i:b_i< \tilde{y}\}|v + |\{i:b_i> \tilde{y}\}| (-v) + |\{i:b_i=\tilde{y}\}| |v|
\end{align*}
and
\begin{align*}
\|(\tilde{x}+u,\tilde{y}+v)-\diag\|_1-\|(\tilde{x},\tilde{y})-\diag\|_1=((\tilde{y}+v)-(\tilde{x}+u))-(\tilde{y}-\tilde{x}) = v-u.
\end{align*}
Together these imply that 
\begin{align*}
f&((\tilde{x}+u,\tilde{y}+v)) - f((\tilde{x},\tilde{y}))\\
&= |\{i:a_i<\tilde{ x}\}| u + |\{i:a_i>\tilde{ x}\}| (-u) + |\{i:a_i=\tilde{x}\}| |u|+ (N-k)(-u)\\
&\quad+ |\{i:b_i< \tilde{y}\}| v + |\{i:b_i> \tilde{y}\}| (-v) + |\{i:b_i=\tilde{y}\}||v| + (N-k)v.
\end{align*}

Since $\tilde{x}$ is the median of  $\{a_1, a_2 \ldots a_k\}$ with $N-k$ copies of $\infty$  we know that  
\[\bigg|(|\{i:a_i>\tilde{ x}\}| + (N-k) ) -|\{i:a_i<\tilde{ x}\}| \bigg| < |\{i:a_i=\tilde{ x}\}|.\] 
This implies that if $u\neq 0$ then \[ |\{i:a_i<\tilde{ x}\}| u + |\{i:a_i>\tilde{ x}\}| (-u) + |\{i:a_i=\tilde{x}\}| |u|+ (N-k)(-u)>0.\] 

Similarly if $v\neq 0$ then \[ |\{i:b_i< \tilde{y}\}| v + |\{i:b_i> \tilde{y}\}| (-v) + |\{i:b_i=\tilde{y}\}| |v| + (N-k)v> 0.\]

Thus $f((\tilde{x}+u, \tilde{y}+v)) > f((\tilde{x},\tilde{y}))$ for $(\tilde{x}+u,\tilde{y} +v)$ sufficiently near, but not equal to, $(\tilde{x},\tilde{y})$, and that $(\tilde{x},\tilde{y})$ is a local minimum. The convexity of $f$ further implies that $(\tilde{x},\tilde{y})$ is the global minimum of $f$ over the domain $\R^{2+}$. Remember we are not including the diagonal as a candidate of locations for the minimum here as $f$ is a function over $\R^{2+}$.

Now suppose that  $(\tilde{x}, \tilde{y})$ lies on or below the diagonal. 
Let $(x,y)\in \R^{2+}$. Then either $x<\tilde{x}$ or $y>\tilde{y}$. Suppose that $x<\tilde{x}$. Let $x' \in (x,\tilde{x})$ with $(x',y) \in \R^{2+}$. Then 
\[f((x,y))-f((x',y))=\sum_{i=1}^k (|x-a_i|- |x'-a_i|) + \sum_{i=k+1}^N (x'-x)> 0\]
%
as $|x-a_i| -|x'-a_i| = x'-x$ whenever $a_i\geq \tilde{x}$ and \[|\{i:a_i\geq\tilde{x}\}|+(N-k) >|\{i:a_i<\tilde{x}\}|\] as $\tilde{x}$ is the median of the $a_i$ and $N-k$ copies of $\infty$. 

%

A similar argument shows that if $y>\tilde{y}$ and $y' \in (\tilde{y},y)$ with $(x,y') \in \R^{2+}$ then $f((x,y))>f((x,y'))$. Thus $f$ decreases as $(x,y)$ travels towards $(\tilde{x},\tilde{y})$ while staying within $\R^{2+}$. By considering limits there is some $t$ such that
$$f((x,y))> \sum_{i=1}^k\|(t,t)-(a_i,b_i)\|_1 +  (N-k)\|(t,t)-\diag\|_1.$$ The observation that $\|(t,t)-(a_i,b_i)\|_1\geq \|\diag-(a_i,b_i)\|_1$ for all $i$, $t$ completes the proof.
\end{proof}

\begin{lemma}\label{lemma:mostdiag}
If $k<N/2$ then
\[\sum_{i=1}^k \|(x,y)-(a_i,b_i)\|_1+ \sum_{i=k+1}^N\|(x,y)-\diag\|_1> \sum_{i=1}^k \|\diag-(a_i,b_i)\|_1\] for every point $(x,y)\in \R^{2+}$
\end{lemma}

\begin{proof}
$\sum_{i=k+1}^N\|(x,y)-\diag\|_1> \sum_{i=1}^k\|(x,y)-\diag\|_1$ as $k<N/2$. Then use the triangle inequality; $\|(x,y)-(a_i,b_i)\|_1+\|(x,y)-\diag\|_1\geq \|\diag-(a_i,b_i)\|_1$.
\end{proof}

Using Proposition \ref{prop:median} and Lemma \ref{lemma:mostdiag} we can characterize the median of a odd sized population of points in $\R^{2+}\cup \Delta$.
\begin{corollary}\label{def:medianselection}
Let $N$ be odd and $S$ the population $\{(a_1, b_1),  \ldots, (a_k, b_k), \Delta, \ldots , \Delta\}$ containing $N-k$ copies of $\Delta$. Let $\tilde{x}$ be the median of $\{a_1, a_2 \ldots a_k\}$ with $N-k$ copies of $\infty$ and let $\tilde{y}$ be the median of $\{b_1, b_2, \ldots , b_k\}$ with $N-k$ copies of $-\infty$.
If $(\tilde{x}, \tilde{y})$ lies above the diagonal then the median of  $S$ is either $(\tilde{x},\tilde{y})$ or $\diag$ (depending on whether $f((\tilde{x},\tilde{y}))$ or $f(\diag)$ is smaller). If $(\tilde{x},\tilde{y})$ lies on or below the diagonal (or is $(\infty, -\infty)$ which philosophically lies below the diagonal) then $\diag$ is the median of $S$. 
\end{corollary}

\begin{figure}[hbt]
\begin{center}
\begin{tikzpicture}[scale=.5]
\draw[->] (0,-1)--(0,7);
\draw[->](-1,0)--(7,0);
\draw[](-1,-1)--(7,7);
\draw plot[mark=square*,mark options={mark size=4pt, fill=white}] coordinates{
(1,4)};
\draw plot[mark=triangle*,mark options={mark size=6pt, fill=white}] coordinates{
(2.4,5)};
\draw plot[mark=diamond*,mark options={mark size=6pt, fill=white}] coordinates{
(1.7, 3.6)};
\draw plot[mark=pentagon*, mark options={mark size=5pt, fill=white}] coordinates{
(0.5, 2)};
\draw (4, 4.5)  circle (4pt);
\draw[dashed] (1.7,-1)--(1.7,7);
\draw[dashed](-1,4)--(7,4);
\fill[](1.7,4) circle (5pt);
\end{tikzpicture}
\end{center}
\caption{The median of the square, triangle and diamond, circle and pentagon points. The order of the $x$ coordinates are $\{\text{pentagon, square, diamond, triangle, circle}\}$ and hence the median is that of the diamond. The order of the $y$ coordinates are $\{$pentagon, diamond, square, circle, triangle$\}$ and hence the $y$ coordinates of the median is that of the square.}
\end{figure}
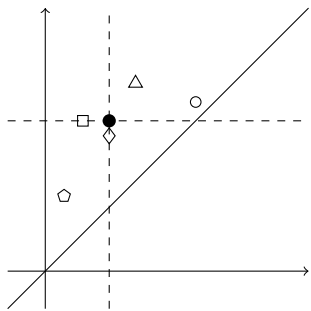

\begin{figure}
\begin{center}
\begin{tikzpicture}[scale=.5]
\draw[->] (0,-1)--(0,7);
\draw[->](-1,0)--(7,0);
\draw[](-1,-1)--(7,7);
\draw plot[mark=square*,mark options={mark size=4pt, fill=white}] coordinates{
(1,4)};
\draw plot[mark=triangle*,mark options={mark size=6pt, fill=white}] coordinates{
(2.4,5)};
\draw plot[mark=diamond*,mark options={mark size=6pt, fill=white}] coordinates{
(1.7, 3.6)};
\draw[dashed] (2.4,-1)--(2.4,7);
\draw[dashed](-1,3.6)--(7,3.6);
\fill[](2.4,3.6) circle (5pt);
\end{tikzpicture}\end{center}
\caption{The median of the square, triangle and diamond points alongside two copies of the diagonal. The order of the $x$ coordinates are $\{\text{square, diamond, triangle, }\infty, \infty\}$ and hence the median is that of the triangle. The order of the $y$ coordinates are $\{-\infty, -\infty$, diamond, square, triangle$\}$ and hence the $y$ coordinates of the median is that of the diamond.}
\end{figure}
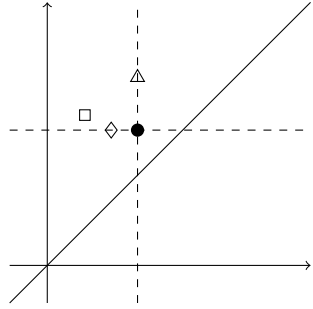

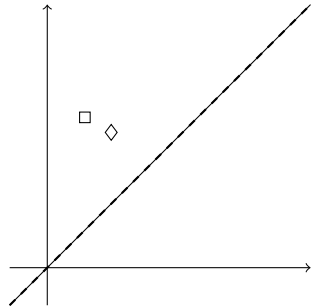
\begin{figure}
\begin{center}
\begin{tikzpicture}[scale=.5]
\draw[->] (0,-1)--(0,7);
\draw[->](-1,0)--(7,0);
\draw[](-1,-1)--(7,7);
\draw[thick,dashed](-1,-1)--(7,7);
\draw plot[mark=square*,mark options={mark size=4pt, fill=white}] coordinates{
(1,4)};
\draw plot[mark=diamond*,mark options={mark size=6pt, fill=white}] coordinates{
(1.7, 3.6)};
\end{tikzpicture}\end{center}
\caption{The median of the square and diamond points alongside three copies of the diagonal. The order of the $x$ coordinates are $\{\text{square, diamond, }\infty, \infty, \infty\}$ and hence the median is ``$\infty$''. The order of the $y$ coordinates are $\{-\infty, -\infty, -\infty$, diamond, square $\}$ and hence the median is $-\infty$. This implies that the median is a copy of the diagonal.}
\end{figure}

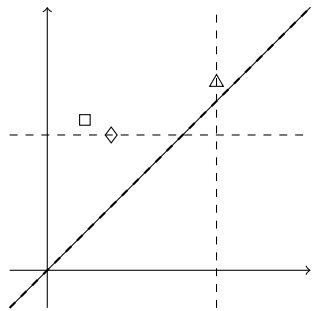
\begin{figure}
\begin{center}
\begin{tikzpicture}[scale=.5]
\draw[->] (0,-1)--(0,7);
\draw[->](-1,0)--(7,0);
\draw[](-1,-1)--(7,7);
\draw plot[mark=square*,mark options={mark size=4pt, fill=white}] coordinates{
(1,4)};
\draw plot[mark=triangle*,mark options={mark size=6pt, fill=white}] coordinates{
(4.5,5)};
\draw plot[mark=diamond*,mark options={mark size=6pt, fill=white}] coordinates{
(1.7, 3.6)};
\draw[thick,dashed](-1,-1)--(7,7);

\draw[dashed] (4.5,-1)--(4.5,7);
\draw[dashed](-1,3.6)--(7,3.6);
\end{tikzpicture}
\end{center}
\caption{The median of the square and diamond points alongside three copies of the diagonal. The order of the $x$ coordinates are $\{\text{square, diamond, triangle }\infty,  \infty\}$ and hence the candidate median has $x$ coordinate of the triangle . The order of the $y$ coordinates are $\{-\infty, -\infty,$ diamond, square, triangle $\}$ and hence the hence the candidate median has $y$ coordinate of the diamond.  However this candidate lies below the diagonal and hence the median is a copy of the diagonal.}
\end{figure}

The median of $S$ is at most two locations as the median of an odd number of (extended) real numbers is unique. In contrast, if $N$ is even then we would have the same uniqueness issues as for sets of real numbers. For example, given the set $X=\{2,3,4,10\}$ every point in the interval $[3,4]$ will minimise the sums distances of points to $X$. However, by convention, we generally say that $3.5$ is the median of $X$. When considering the median of an even number of points in $\R^{2+}\cup \Delta$, instead of a unique point in $\R^{2+}\cup\Delta$ we have a rectangle of options, in which every point would minimize $f_S$. We could in theory adopt an analogous convention by using the barycenter of this rectangle. This may be application dependent. However, for clarity of exposition, and to make the statements of theorems much cleaner, we will be restricting our attention to $N$ odd.

\subsection{Characterizing the median(s) of sets of diagrams} 
In \cite{turner2014frechet} there is a complete characterization of the local minima of $F_2$ when the observations are finitely many persistence diagrams each with only finitely many off diagonal points. However, the proof used the Alexandrov space structure of $(\D, d_2)$ and hence we cannot adapt it to characterize the local minima of $F_1$. Here we use an alternative approach to prove analogous necessary and sufficient conditions for a persistence diagram to be a local minimum of $F_1$.

This characterization of local minima of $F_2$ was rephrased in \cite{munch2015probabilistic} in terms of selections, groupings and optimal pairings. Here we will characterize the local minima of $F_1$ with this same terminology. Given a set of diagrams $X_1,\cdots,X_N$, a \textit{selection} is a choice of one point from each diagram (where that point can be $\diag$). A \textit{grouping} is a set of selections so that every  off diagonal   point of every diagram is part of exactly one selection. Our notation will be as follows. If $S$ is a selection then let $m_S$ be the median of that selection (chosen to be the off diagonal point if not unique). A grouping $G$ of $X_1, \ldots X_N$ is the set of selections $G = \{S_j\}$. Let $m(G)$ be the persistence diagram which contains $\{m_{S_j}:S_j\in G\}$. Each grouping $G$ produces a candidate $m(G)$ for the median. We will show that any median of $X_1, \ldots X_N$ must be $m(G)$ for some grouping $G$ of $X_1, \ldots X_N$.

We will consider two groupings as equivalent if they only differ by selections containing only copies of the diagonal. Note that equivalent groupings produce the same persistence diagram as their median candidate. This implies that in theorems and algorithms we can restrict to groupings where each selection contains at least one off-diagonal point.

Since any minimum is also a local minimum, we will focus now on characterizing the local minima of $F_1$. We will first show that if $Y$ is the median of $\{X_1, \ldots, X_N \}$ then $Y=m(G)$ whenever $G$ is an appropriate grouping constructed using optimal bijections $\phi_i:Y\to X_i$.

\begin{theorem}\label{thm:median}
Let $X_1, \ldots, X_N$ be persistence diagrams in $\D^{(0,0)}$. Let $Y=\{y_j\}\in \D^{(0,0)}$. For each $i$ let $\phi_i:Y \to X_i$ be an optimal bijection between $Y$ and $X_i$ using the distance function $d_1$. For each $y\in Y$ we have a selection $\{\phi_i(y)\}$ (to make this well defined we think of the copies of $\Delta$ when $\phi_i^{-1}(x_j)=\Delta$ as each disjoint). Let $G$ be the grouping $\{ \{\phi_i(y)\}\colon y\in Y\}$. 

If $Y$ is a local minimum of $F_1:Z \mapsto \frac1N \sum_{i=1}^N d_1(X_i,Z)$ then $Y=m(G)$.
\end{theorem}

\begin{proof}
Suppose that $Y\neq m(G)$ and thus $y\neq m_{\{\phi_i(y)\}}$ for some $y\in Y$. We need to split into cases depending on whether or not $m_{\{\phi_i(y)\}}$ is the diagonal.
%

If $y=\diag$ then  $\{\phi_i(y)\}$ contains at most one off diagonal point. By Lemma \ref{lemma:mostdiag} we know that $m_{\{\phi_i(y)\}}=\diag$.

Suppose now that $y\neq \diag$ and that more that half of $\{\phi_i(y)\}$ are copies of the diagonal. As $z$ moves from $y$ to the closest point on the diagonal $\sum_{\{i:\phi_i(y)\neq \diag\}} \|z-\phi_i(y)\|_1$ increases less than $\sum_{\{i:\phi_i(y)=\diag\}} \|z-\diag\|_1$ decreases and hence $\sum_i \|z-\phi_i(y)\|_1$ must be decreasing. Let $Z= \{z\}\cup Y\backslash y$. $F_1(Z)$ decreases as $z$ moves from $y$ towards the diagonal. Thus $Y$ cannot be a local minimum. 

Finally suppose that $y\neq \diag$ and more than half the points of $\{\phi_i(y)\}$  are off the diagonal. Consider the point  $(\tilde{x},\tilde{y})\in \R^2$ introduced in Proposition \ref{prop:median}. If $(\tilde{x},\tilde{y})$ lies above the diagonal then by Proposition \ref{prop:median} we know that $\sum_i \|z-\phi_i(y)\|_1$ decreases as $z$ travels along a straight line from $y$ to $m_{\{\phi_i(y)\}}$. 

If $(\tilde{x}, \tilde{y})$ lies on or below the diagonal then the proof of Proposition 
\ref{prop:median} shows that $\sum_i \|z-\phi_i(y)\|_1$ decreases as $z$ moves from $y$ to $\diag=m_{\{\phi_i(y)\}}$. In both cases this implies that $F_1$ would also be decreasing as $z$ travels from $y$ towards $m_{\{\phi_i(y)\}}$. Having now exhausted all the possibilities we conclude $Y$ is not a local minimum.
\end{proof}

In the above theorem we made no assumption about the uniqueness of the optimal bijections $\phi_i:Y\to X_i$. Instead this necessary condition holds for \emph{any} set of optimal bijections. This is slightly different to the scenario of the mean. In \cite{turner2014frechet} it was shown that if $Y$ was a local minimum of $F_2$ the $\phi_i$ were essentially unique (up to relabelling points in the persistent diagrams at the same location in $\R^{2+}$. However this uniqueness does not hold for the local minima of $F_1$. This is because shifting an observation $a_i$ within a population $\{a_1, \ldots, a_N\}$ of real numbers does not affect the median unless $\sgn(a_i-median\{a_1, \ldots, a_N\})$ changes. Figure \ref{fig:notuniquebijections} provides an explicit example.

\begin{center}
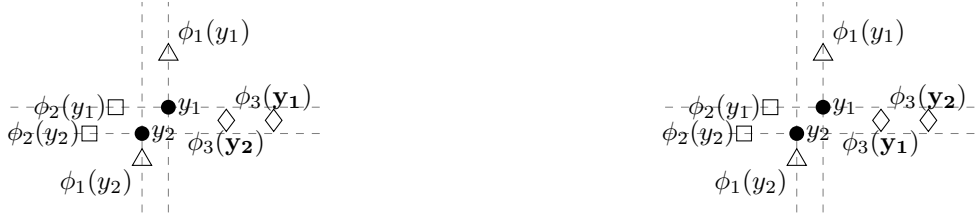
\begin{figure}[hbt]
\begin{minipage}{0.455\linewidth}\centering
\begin{tikzpicture}[scale=0.7]

\draw plot[mark=square*,mark options={mark size=4pt, fill=white}] coordinates{(0,5)}node[left]{$\phi_2(y_1)$};
\draw plot[mark=square*,mark options={mark size=4pt, fill=white}] coordinates{(-0.5,4.5)}node[ left]{$\phi_2(y_2)$};

\draw plot[mark=triangle*,mark options={mark size=6pt, fill=white}] coordinates{
(0.5,4)} node[below left]{$\phi_1(y_2)$};
\draw plot[mark=triangle*,mark options={mark size=6pt, fill=white}] coordinates{
(1,6)} node[above right]{$\phi_1(y_1)$};

\draw plot[mark=diamond*,mark options={mark size=6pt, fill=white}] coordinates{
(3,4.75)}node[above]{$\phi_3(\mathbf{y_1})$};
\draw plot[mark=diamond*,mark options={mark size=6pt, fill=white}] coordinates{
(2.1,4.75)}node[below]{$\phi_3(\mathbf{y_2})$} ;


\draw[dashed, opacity=0.5] (1,3)--(1,7);
\draw[dashed, opacity=0.5](-2,5)--(4,5);
\fill[](1,5)node[right]{$y_1$} circle (4pt);

\draw[dashed, opacity=0.5] (0.5,3)--(0.5,7);
\draw[dashed,opacity=0.5](-2,4.5)--(4,4.5);
\fill[](0.5, 4.5)node[right]{$y_2$}  circle (4pt);
\end{tikzpicture}
\subcaption{A grouping where the $y_2$ is the median of the selection containing the diamond on the \emph{left}.}
\end{minipage}%
\qquad 
\begin{minipage}{0.455\linewidth}\centering
\begin{tikzpicture}[scale=0.7]

\draw plot[mark=square*,mark options={mark size=4pt, fill=white}] coordinates{(0,5)}node[left]{$\phi_2(y_1)$};
\draw plot[mark=square*,mark options={mark size=4pt, fill=white}] coordinates{(-0.5,4.5)}node[ left]{$\phi_2(y_2)$};

\draw plot[mark=triangle*,mark options={mark size=6pt, fill=white}] coordinates{
(0.5,4)} node[below left]{$\phi_1(y_2)$};
\draw plot[mark=triangle*,mark options={mark size=6pt, fill=white}] coordinates{
(1,6)} node[above right]{$\phi_1(y_1)$};

\draw plot[mark=diamond*,mark options={mark size=6pt, fill=white}] coordinates{
(3,4.75)}node[above]{$\phi_3(\mathbf{y_2})$};
\draw plot[mark=diamond*,mark options={mark size=6pt, fill=white}] coordinates{
(2.1,4.75)}node[below]{$\phi_3(\mathbf{y_1})$};


\draw[dashed, opacity=0.5] (1,3)--(1,7);
\draw[dashed, opacity=0.5](-2,5)--(4,5);
\fill[](1,5)node[right]{$y_1$} circle (4pt);

\draw[dashed, opacity=0.5] (0.5,3)--(0.5,7);
\draw[dashed,opacity=0.5](-2,4.5)--(4,4.5);

\fill[](0.5, 4.5)node[right]{$y_2$}  circle (4pt);
\end{tikzpicture}
\subcaption{A grouping where the $y_2$ is the median of the selection containing the diamond on the \emph{right}.}
\end{minipage}%

%
%
%
%
\caption{$Y$ (black) is a local minimum of $F_1$. It is the unique median of the ``triangle'', ``square'' and ``diamond'' diagrams. There are two optimal bijections $\phi_3$ from $Y$ to the ``diamond'' diagram, creating two groupings $G_1$ and $G_2$. Here $Y=m(G_1)=m(G_2)$.}\label{fig:notuniquebijections}
\end{figure}
\end{center}

The following is a sufficient condition for a persistence diagram to be a local minimum of the function $F_1:Z \mapsto \frac1N \sum_{i=1}^N d_1(X_i,Z)$ when we restrict to input diagrams $X_i$ containing only finitely many off diagonal points.

\begin{theorem}
Let $X_1, \ldots, X_N\in \D^{(0,0)}$ be persistence diagrams with only finitely many  off diagonal  points. Let $Y=\{y_j\}\in \D$. Suppose that whenever $\phi_i:Y \to X_i$ are optimal bijections 
\begin{enumerate}
\item $y=m_{\{\phi_i(y)\}}$ whenever $y\in Y$ is off diagonal, and 
\item for any selection $S=\{x_1, x_2, \ldots x_N\}$ such that $x_i \in X_i$ and $\phi_i^{-1}(x_i)=\Delta$, we have $m_S=\Delta$.
\end{enumerate}
Then $Y$ is a local minimum of $F_1:Z \mapsto \frac1N \sum_{i=1}^N d_1(X_i,Z)$
%
\end{theorem}
\begin{proof}

Let $\phi_i:Y \to X_i$ be optimal bijections. Assume that $y=m_{\{\phi_i(y)\}}$ whenever $y\in Y$ is off diagonal. Since each of the $X_i$ contain only finitely many off diagonal points there can only be finitely many selections of the $\{X_1, X_2, \ldots X_N\}$ containing some off diagonal point. Thus there can only be finitely many off diagonal points in $Y$.

Suppose that $Y$ is not a local minimum. Then there exists a sequence $Y_n$ that converges to $Y$ such that $F_1(Y_n)<Y$ for all $n$. For each $Y_n$ fix optimal bijections $\psi_n:Y \to Y_n$. Fix an off diagonal point $y\in Y$. Since $\|y-\Delta\|>0$ and $d_1(Y, Y_n) \to 0$ we know $\psi_n(y) \neq \Delta$ for large enough $n$. 

For each $i$ choose optimal bijections $\phi_i^n:Y_n \to X_i$. Consider the sequence $(\phi_i^n \circ \psi_n)(y) \in X_i$. Since $X_i$ has only finitely many off diagonal points this sequence has a constant subsequence (here we think of the sequence containing only copies of the diagonal as constant).
By taking subsequences of subsequences we can find a subsequence $\hat{Y_l}$ of $Y_n$ such that such that $(\phi_i^l\circ \psi_l)(y)$ is constant for all off diagonal $y\in Y$ and all $i$. 

Construct $\beta_i\colon Y\to X_i$ by $\beta_i(y)=\phi_i^l\circ \psi_l(y)$ and $\beta_i(x)=\diag$ for any remaining unmatched points $x\in X_i$ We will show these $\beta_i\colon Y \to X_i$ are optimal bijections. 

For each bijection $\tau:A \to B$ let $\C(\tau)=\sum_{a\in A}\|a-\tau(a)\|_1$ denote the 1-Wasserstein transportation cost via the bijection $\tau$. Thus $\tau:A\to B$ is an optimal bijection if and only if $d_1(A,B) =\C(\tau)$.

Suppose that $\beta_i:Y \to X_i$ is not optimal. This implies there is some bijection $\alpha:Y \to X_i$ and $\epsilon>0$ with 
the $\C(\alpha)< \C(\beta_i) -\epsilon.$ 

Since $\lim_{l\to \infty}\hat{Y_l} =Y$ there is some $l$ and some bijection $\psi_l:\hat{Y_l} \to Y$ such that $\C(\psi_l)< \epsilon/3$. Let $\hat{\beta_i}:\hat{Y_l} \to X_i$ be the  transportation plan by first transporting $Y_l$ via $\psi_l$ to $Y$ and then transporting via $\alpha$ to $X_i$. By contruction
$\C({\hat{\beta_i}})\leq \C(\alpha) + \C(\psi_l) <\C(\beta_i) -2\epsilon/3.$
But at the same time, by considering $\beta_i$ as the composition of the transportation plans of $\hat{\beta_i}$ and $\psi_l$, we know
$\C(\beta_i)\leq \C(\hat{\beta_i}) + \C(\psi_l) < \C(\hat{\beta_i}) + \epsilon/3.$

Together these inequalities imply that $\C(\beta_i) < \C(\beta_i) -\epsilon/3$ which is impossible. Thus the $\beta_i=\phi_i^l\circ \psi_l\colon Y \to X_i$ are optimal bijections for all $i$.

Now 
$F_1(Y)= \frac1N\sum_{i=1}^N \left(\sum_{\{y\in Y\colon y\neq \Delta\}} \|y-\beta_i(y)\|_1 +  \sum_{\{x\in X_i\colon \beta_i^{-1}(x)=\Delta\}} \|x-\Delta\|_1\right)$
 and 
 \begin{align*}
F_1(Y_l)&= \frac1N\sum_{i=1}^N \left(\sum_{\{\hat{y}\in Y\colon \psi_l(\hat{y})\neq \Delta\}} \|\psi_l(\hat{y})-\phi_i^l(y)\|_1 + \sum_{\{\hat{y}\in Y_l\colon \psi_l(\hat{y})=\Delta\}} \|\hat{y} -\phi_i^l(\hat{y})\|_1\right)\\
&= \frac1N\sum_{i=1}^N \left( \sum_{\{y\in Y\colon y\neq \Delta\}} \|\psi_l^{-1}(y)-\beta_i(y)\|_1 + \sum_{\{\hat{y}\in Y_l\colon \psi_l(\hat{y})=\Delta\}} \|\hat{y} -\phi_i^l(\hat{y})\|_1\right).
\end{align*}

By assumption $y=m_{\{\beta_i(y)\}}$ and hence \[\sum_{i=1}^N \sum_{\{y\in Y\colon y\neq \Delta\}} \|y-\beta_i(y)\|_1 \leq\sum_{i=1}^N  \sum_{\{y\in Y\colon y\neq \Delta\}} \|\psi_l^{-1}(y)-\beta_i(y)\|_1.\] Thus for $F_1(Y_l)<F_1(Y)$ to hold it must be true that  \[ \sum_{i=1}^N  \sum_{\{x\in X_i\colon \beta_i^{-1}(x)=\Delta\}} \|x-\Delta\|_1>  \sum_{i=1}^N \sum_{\{\hat{y}\in Y_l\colon \psi_l(\hat{y})=\Delta\}} \|\hat{y} -\phi_i^l(\hat{y})\|_1.\]
Since $\{\hat{y}\in Y_l\colon \psi_l(\hat{y})=\Delta\} \subset \{\hat{y}\in Y_l\colon (\beta_i^{-1}\circ \phi_i^l)(\hat{y})=\Delta\}$ we further know
\[ \sum_{i=1}^N  \sum_{\{x\in X_i\colon \beta_i^{-1}(x)=\Delta\}} \|x-\Delta\|_1>  \sum_{i=1}^N  \sum_{\{x\in X_i\colon \beta_i^{-1}(x)=\Delta\}} \|x-(\phi_i^l)^{-1}(x)\|_1.\]
This implies that there is a selection $S=\{x_1, x_2, \ldots x_N\}$ such that $x_i \in X_i$, $\beta_i^{-1}(x_i)=\Delta$ and $m_S\neq\Delta$, contradicting our second condition.
\end{proof}

Theorem \ref{thm:median} provides us with an (admittedly very slow) algorithm to find the median. We can consider the set of all groupings $G$ up to equivalence, and their corresponding candidates $m(G)$. The median is one of these $m(G)$ so we only need to compare the $F_1(m(G))$ over all groupings $G$. We only have to use the necessary condition as it is always allowable to check (finitely many) extra options when looking for global minima as long as we know we have all the local minima in our list to check.

Alternatively we could use a gradient descent algorithm analogous to that used in \cite{turner2014frechet}. The only modifications needed are replacing the optimal pairing using $d_2$ with optimal pairings using $d_1$ and replacing the means of the selections with the medians of the selections. This algorithm will terminate in finite time as at each iteration the cost function $F_1$ decreases and uses a new grouping (of which there are only finitely many).  This would not guarantee finding the global minimum  but rather will terminate at a local minimum. Running multiple times from different initial locations can improve the estimate. 

\section{Comparing the median and the mean}

\subsection{Robustness of the median}


For real numbers, the median is a robust measure of central tendency, while the mean is not.  One measure of robustness is the breakdown point, which is of a bound on the proportion of incorrect observations (e.g. arbitrarily large observations) that an estimator can handle before giving an incorrect (e.g. arbitrarily large) result. The median has a breakdown point of $50\%$, while the mean has a breakdown point of $0\%$ (a single large observation can throw it off). In this section we will prove that the breakdown points for medians and mean of persistence diagrams are the same as those of real numbers. For persistence diagrams we replace ``arbitrarily large'' with ``arbitrarily far finite distance away.'' 

\begin{lemma}
The breakdown point for the mean of a population of persistence diagrams lying in $\D^{(0,0)}$ is $0\%$.
\end{lemma}
\begin{proof}
Let $X_1, X_2, \ldots X_N\in\D^{(0,0)}$ with mean $Y$. There is some $M>0$ such that every point in $Y$ is at most distance $M$ from the diagonal. Let $\tilde{X}$ be a diagram with a single off diagonal point $p=(0, \sqrt{2}(KN+ MN))$. Observe that $p$ is distance  $KN+MN$ from the diagonal.

Let $Z$ be a mean of $\{\tilde{X}, X_2, \ldots X_N\}$. Using the characterisation of the mean, $Z$ must contain a point at least distance $(KN+MN)/N=K+M$ from the diagonal. This implies that $d_2(Z,Y)\geq K$. By choosing $K$ arbitrarily large we can ensure $Y$ and $Z$ are arbitrarily far away. 
\end{proof}

The breakdown point for the median of a population of persistence diagrams lying in $\D^{(0,0)}$ is $50\%$. Let $\emptyset$ denote the persistence diagram only containing copies of the diagonal.

\begin{theorem}
Let $X_1, X_2, \ldots X_{n+1}\in \D^{(0,0)}$ each with only finitely many  off diagonal points. Then exists a constant $M$ (dependent on $X_1, X_2, \ldots X_n$) such that for any $X_{n+2}, \ldots X_{2n+1} \in \D^{(0,0)}$, any median of $\{X_1, X_2, \ldots X_{n+1}, X_{n+2}, \ldots X_{2n+1}\}$ is of distance at most $M$ from the persistence diagram only containing copies of the diagonal.
\end{theorem}
\begin{proof}
Set our bound $M$ as
\[\max_{\text{ Groupings }G\text{ of } X_1,X_2,\ldots X_{n+1}}\sum_{\text{selections }s\in G} \max\{y\text{-coords in }s\} -  \min\{x\text{-coords in }s\} \]
which only depends on the diagrams $\{X_1, X_2, \ldots X_{n+1}\}$.

Let $Y=\{(a_j, b_j)\}$ be a median of $\{X_1, X_2, \ldots X_{n+1}, X_{n+2}, \ldots X_{2n+1}\}$. From optimal bijections $\phi_i:Y \to X_i$ denote the coordinates of $\phi_i((a_j, b_j))$ as $(x_{i,j}, y_{i,j})$ (writing $(\infty, -\infty)$ when $\phi_i((a_j,b_j))=\Delta$). Since $a_j$ is the median of $\{x_{1,j}, x_{2,j}, \ldots x_{2n+1, j}\}$ we know that 
$a_j \geq \min \{x_{1,j}, x_{2,j}, \ldots x_{n+1, j}\}.$
Similarly $b_j \leq \max \{y_{1,j}, y_{2,j}, \ldots y_{n+1, j}\}.$
This implies that  for each $j$, we have the bound 
$b_j-a_j\leq \max \{y_{1,j}, y_{2,j}, \ldots y_{n+1, j}\} - \min \{x_{1,j}, x_{2,j}, \ldots x_{n+1, j}\}.$
Thus 
\[d_1(Y, \emptyset)=\sum_{j} b_j-a_j \leq \sum_j  \max \{y_{1,j}, y_{2,j}, \ldots y_{n+1, j}\} - \min \{x_{1,j}, x_{2,j}, \ldots x_{n+1, j}\}.\]

From our characterisation of the median of a set of persistence diagrams we know $Y= m(\hat{G})$ for some grouping $\hat{G}$ of $\{X_1, X_2, \ldots X_{2n+1}\}$. Let $\hat{G}^{res}$ be the restriction of $\hat{G}$ to the subset of diagrams $\{X_1, X_2, \ldots X_{n+1}\}$.

By construction 
\begin{align*}
M&\geq \sum_{\text{selections }s\in \hat{G}^{res}} \max\{y\text{-coordinates in }s\} -  \min\{x\text{-coordinates in }s\}\\
&=\sum_j  \max \{y_{1,j}, y_{2,j}, \ldots y_{n+1, j}\} - \min \{x_{1,j}, x_{2,j}, \ldots x_{n+1, j}\}\\
&\geq d_1(Y, \emptyset)
\end{align*}

\end{proof}

\subsection{Number of points in the mean compared to the median}\label{sec:number points}

One qualitative difference between the mean and the median is the presence or absence of points with small persistence. In some applications, such as when we have point cloud samples of an underlying shape of interest, these are heuristically the result of noise. The mean of a selection containing at least one point off the diagonal is a point off the diagonal.  It is possible for the mean of $N$ diagrams each with $K$ points to contain $NK$ off diagonal points. In comparison the median of any selection with more than half copies of the diagonal will always be a copy of the diagonal.  In the big picture this can add up to lots of extra points off the diagonal in the mean persistence diagram when compared to the median.

\begin{lemma}
Let $X_1, \ldots X_N$ be persistence diagrams such that the average number of off diagonal points in the $X_i$ is $K$. If $Y$ is a median of the $X_i$ then $Y$ has less than $2K$ points off the diagonal.
\end{lemma}
\begin{proof}
Let $y_1, y_2, \ldots y_n$ be the off diagonal points in $Y$. Let $\phi_i$ be optimal bijections between $Y$ and the $X_i$. By Theorem \ref{thm:median}  we know that $y_j$ is the median of $\{\phi_i(y_j)\}$ for each $j$. By Lemma \ref{lemma:mostdiag} we know that for each $j$ the sets $\{\phi_i(y_j)\}$ must contain at least $(N+1)/2$ off diagonal points. This implies that $\cup_j \{\phi_i(y_j)\}$ must contain at least $(N+1)n/2$ points.

Since the combined of total of all the off diagonal points in the $X_i$ is $NK$ we can conclude that $(N+1)n/2\leq NK$ and hence $n<2K$.
\end{proof}

We illustrate this significant advantage of the median with a simulated geometric example. We generated point clouds of the unit circle by drawing $25$ points from the uniform measure on the unit circle convoluted with Gaussian noise with variance $\sigma^2$. We then build the $H_1$ persistence diagrams from the corresponding Rips filtration of this point cloud  (described in the appendix). For the $5$ persistence diagrams thus produced we then computed the mean and the median. These are illustrated in Figure \ref{fig:circle}. 

Since the underlying shape of interest is a circle there should be one point in each persistence diagram far from the diagonal corresponding to the $H_1$ class of the circle, alongside extra ``noisy'' points near the diagonal (with more as the noise parameter in the sampling process increases). In the simulated data the mean and median each have one point far from the diagonal (and these are close to each other) but with larger noise the mean diagram has more extra points near the diagonal than the median diagram.

\begin{center}
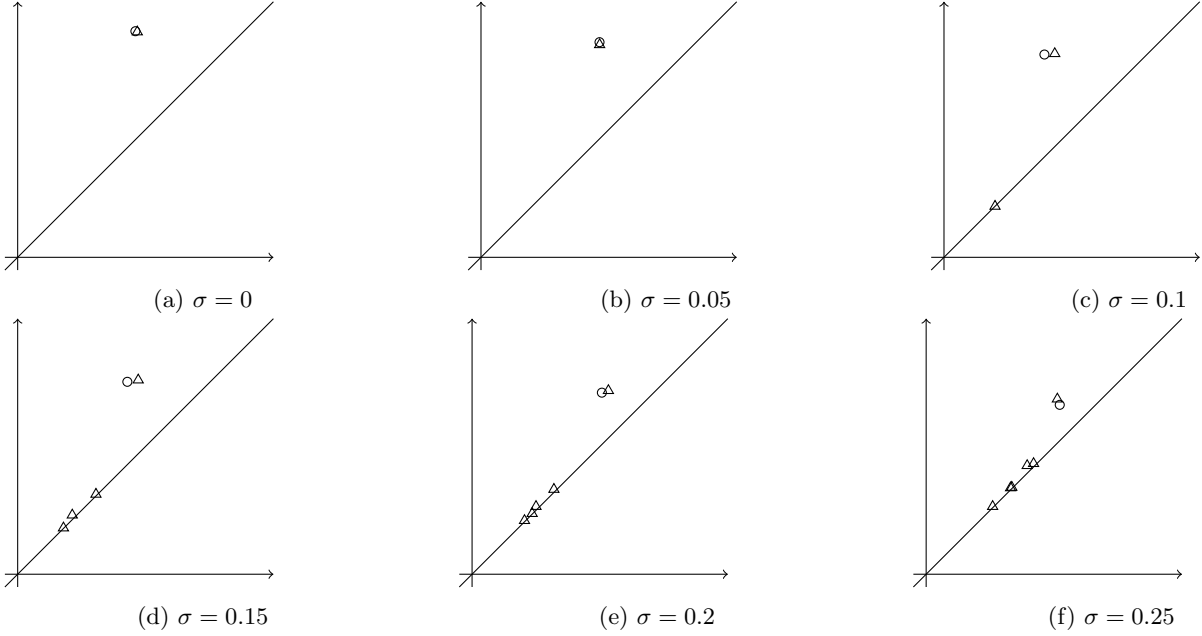
\begin{figure}[h!]
\begin{subfigure}{0.3\textwidth}
\begin{tikzpicture}[scale=1.7]
\draw []   (-0.1,-0.1) -- (2,2);
\draw [->] (-0.1, 0) -- (2,0);
\draw [->] ( 0,-0.1) -- (0,2);
\draw[black] plot[only marks, mark=o, mark size=1pt]coordinates{(.920616941294,1.76959175489) };
\draw[black] plot[only marks, mark=triangle, mark size=1.3pt]coordinates{(.93538,	1.764) };
\end{tikzpicture}
\caption{$\sigma=0$}
\end{subfigure} 
\hspace{0.03\textwidth}
\begin{subfigure}{0.3\textwidth}
\begin{tikzpicture}[scale=1.7]
\draw []   (-0.1,-0.1) -- (2,2);
\draw [->] (-0.1, 0) -- (2,0);
\draw [->] ( 0,-0.1) -- (0,2);
\draw[black] plot[only marks, mark=o, mark size=1pt]coordinates{(0.92689902682,	1.68373772005) };
\draw[black] plot[only marks, mark=triangle, mark size=1.3pt]coordinates{(0.928,			1.664) };
\end{tikzpicture}
\caption{$\sigma=0.05$}
\end{subfigure} \hspace{0.03\textwidth}
\begin{subfigure}{0.3\textwidth}
\begin{tikzpicture}[scale=1.7]
\draw []   (-0.1,-0.1) -- (2,2);
\draw [->] (-0.1, 0) -- (2,0);
\draw [->] ( 0,-0.1) -- (0,2);
\draw[black] plot[only marks, mark=o, mark size=1pt]coordinates{(0.785416370044,	1.58704974453) };
\draw[black] plot[only marks, mark=triangle, mark size=1.3pt]coordinates{(0.8662,1.594)(0.3992,0.3993) };
\end{tikzpicture}
\caption{$\sigma=0.1$}
\end{subfigure} \hspace{0.03\textwidth}
\begin{subfigure}{0.3\textwidth}
\begin{tikzpicture}[scale=1.7]
\draw []   (-0.1,-0.1) -- (2,2);
\draw [->] (-0.1, 0) -- (2,0);
\draw [->] ( 0,-0.1) -- (0,2);
\draw[black] plot[only marks, mark=o, mark size=1pt]coordinates{(0.857607927762,	1.5058382678) };
\draw[black] plot[only marks, mark=triangle, mark size=1.3pt]coordinates{(0.9428,	1.52) (.427, .463) (0.6122, .6258) (0.358, .362) };

\end{tikzpicture}
\caption{$\sigma=0.15$}
\end{subfigure} \hspace{0.03\textwidth}
\begin{subfigure}{0.3\textwidth}
\begin{tikzpicture}[scale=1.7]
\draw []   (-0.1,-0.1) -- (2,2);
\draw [->] (-0.1, 0) -- (2,0);
\draw [->] ( 0,-0.1) -- (0,2);
\draw[black] plot[only marks, mark=o, mark size=1pt]coordinates{(1.0150966146,	1.42097064809) };
\draw[black] plot[only marks, mark=triangle, mark size=1.3pt]coordinates{(1.066,1.4372)(0.64, 0.664) (0.41,0.42) (0.47, 0.475)(0.50, .53)};
\end{tikzpicture}
\caption{$\sigma=0.2$}
\end{subfigure} 
\hspace{0.03\textwidth}
\begin{subfigure}{0.3\textwidth}\begin{tikzpicture}[scale=1.7]
\draw []   (-0.1,-0.1) -- (2,2);
\draw [->] (-0.1, 0) -- (2,0);
\draw [->] ( 0,-0.1) -- (0,2);
\draw[black] plot[only marks, mark=o, mark size=1pt]coordinates{(1.04471840956,	1.32566034849) };
\draw[black] plot[only marks, mark=triangle, mark size=1.3pt]coordinates{(0.52,0.53)(.66,.676) ( .67,.68) (.84 ,.866)(.79, .85)(1.026,1.371)};
\end{tikzpicture}
\caption{$\sigma=0.25$}
\end{subfigure} 
\caption{
For each standard deviation $\sigma$ we randomly generated five noisy point clouds of the circle each  $25$ points drawn i.i.d. from the convolution of the uniform measure on the unit circle convolved with Gaussian noise with standard deviation $\sigma$. From these five point clouds we constructed five Rips filtrations and their $H_1$ persistence diagrams. The corresponding median diagram is then depicted using circles and the mean diagram using triangles.}\label{fig:circle}
 \end{figure}
\end{center}

\subsection{Discontinuities and non-uniqueness of the mean and the median}\label{sec:cont}

Unfortunately both the mean and median are neither continuous nor always unique. There can be  a discontinuities when the grouping $G$ which provides us with the optimal candidate for the mean or the median switches. This is illustrated in the Figures \ref{fig:MedianNotCont} and \ref{fig:MeanNotCont}. 

In these examples we have three diagrams, one consists only of copies of the diagonal, one containing off diagonal points denoted by squares, and the other denoted by  triangles. In this example as $z$ increases in the squares diagram travels across the optimal grouping changes from $\{\{x_1, (1,z),\diag\},\{x_2, \diag, \diag\}\}$ to $\{\{x_1, \diag, \diag\},\{x_2, (1,z), \diag\}\}$ leading to a discontinuity to both the mean and the median (note that the value of $z$ where the switch occurs is different for the mean and median). At the time it switches both groupings are equally optimal and hence we have non-uniqueness.

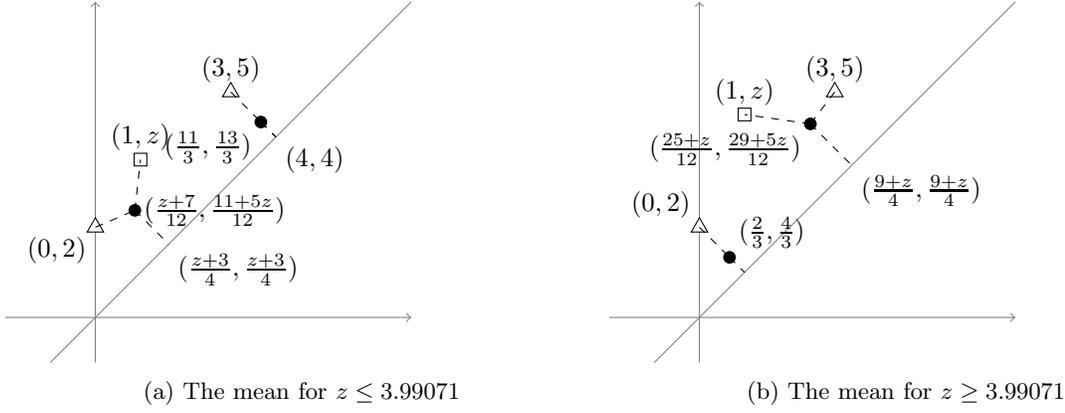
\begin{figure}[hbt]
\begin{subfigure}{0.45\textwidth}

\begin{tikzpicture}[scale=0.6]
\draw[gray, ->] (0,-1)--(0,7);
\draw[gray, ->](-2,0)--(7,0);
\draw[gray](-1,-1)--(7,7);

\draw plot[mark=triangle*,mark options={mark size=6pt, fill=white}]coordinates{
(0,2)} node [below left]  {$(0,2)$};
\draw plot[mark=triangle*,mark options={mark size=6pt, fill=white}]coordinates{
(3,5)} node [above]  {$(3,5)$};

\draw plot[mark=square*,mark options={mark size=4pt, fill=white}]coordinates{
(1, 3.5)} node [above]  {$(1,z)$};


\draw[dashed] (0,2)--(0.875,2.375);
\draw[dashed](0.875,2.375)--(1,3.5);

\draw[dashed] (3,5)--(11/3,13/3);
\draw[dashed] (11/3,13/3)--(4,4) node[below right]{ $(4,4)$};
\draw[dashed] (11/3,13/3)--(4,4);
\fill[] (11/3,13/3) node[below left] {$(\frac{11}{3}, \frac{13}{3})$} circle (4pt);
\draw[dashed](0.875,2.375)--(1.625,1.625) node [below right]{$(\frac{z+3}{4}, \frac{z+3}{4})$};
\fill[](0.875,2.375) node [right]{$(\frac{z+7}{12},\frac{11+5z}{12})$} circle (4pt);
\end{tikzpicture}
\caption{The mean for $z\leq 3.99071$}
\end{subfigure}
\begin{subfigure}{0.45\textwidth}
\begin{tikzpicture}[scale=0.6]
\draw[gray,->] (0,-1)--(0,7);
\draw[gray,->](-2,0)--(7,0);
\draw[gray](-1,-1)--(7,7);


\draw plot[mark=triangle*,mark options={mark size=6pt, fill=white}]coordinates{
(0,2)} node [above left]  {$(0,2)$};
\draw plot[mark=triangle*,mark options={mark size=6pt, fill=white}]coordinates{
(3,5)} node [above]  {$(3,5)$};
\draw plot[mark=square*,mark options={mark size=4pt, fill=white}]coordinates{
(1, 4.5)} node [above]  {$(1,z)$};

\draw[dashed] (0,2)--(2/3,4/3);
\draw[dashed](2/3,4/3)--(1,1);
\fill[] (2/3,4/3) node [above right] {$(\frac{2}{3}, \frac{4}{3})$} circle (4pt);;

\draw[dashed] (3,5)--(29.5/12, 51.5/12);
\draw [dashed](29.5/12, 51.5/12)--(3.375,3.375) node [below right] {$(\frac{9+z}{4},\frac{9+z}{4})$};
\draw[dashed](29.5/12, 51.5/12)--(1,4.5);
\fill[](29.5/12, 51.5/12) node [below left]{$(\frac{25+z}{12}, \frac{29 + 5z}{12})$} circle (4pt);
\end{tikzpicture}
\caption{The mean for $z\geq 3.99071$}\label{fig:meannotunique2}
\end{subfigure}
\caption{We have three diagrams, one consists only of copies of the diagonal, one containing off diagonal points denoted by squares, and the other containing off diagonal points denoted by  triangles. 
 In (a) $F_2(\text{circles}) = \frac{8639 - 3995 z + 1268 z^2}{6534}$ 
and in (b) $F_2(\text{circles}) = \frac{191 -58z + 7z^2}{36}$.
When $z<3.99070$ the optimal grouping is $\{(0,2), (1,z), \diag \}$ and $\{(3,5), \diag, \diag\}$  (used in (a)). When $z>3.99072$ then the optimal grouping is $\{(0,2), \diag, \diag\}$ and $\{(3,5),  (1,z), \diag\}$ (used in (b)). Both groupings are optimal when $z\simeq 3.99071$ and as a result we do not have a unique mean.}\label{fig:MeanNotCont}
\end{figure}

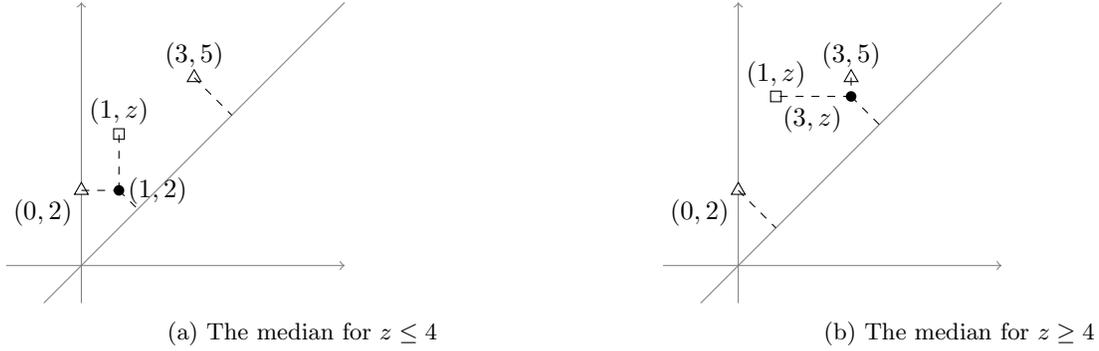
\begin{figure}[h]
\begin{subfigure}{0.45\textwidth}
\begin{tikzpicture}[scale=0.5]
\draw[gray, ->] (0,-1)--(0,7);
\draw[gray, ->](-2,0)--(7,0);
\draw[gray](-1,-1)--(7,7);
%
%

\draw plot[mark=triangle*,mark options={mark size=6pt, fill=white}]coordinates{
(0,2)} node [below left]  {$(0,2)$};
\draw plot[mark=triangle*,mark options={mark size=6pt, fill=white}]coordinates{
(3,5)} node [above]  {$(3,5)$};

\draw plot[mark=square*,mark options={mark size=4pt, fill=white}]coordinates{
(1, 3.5)} node [above]  {$(1,z)$};

\draw[dashed] (0,2)--(1,2);
\draw[dashed](1,2)--(1,3.5);

\draw[dashed] (3,5)--(4,4);
\draw[dashed] (1,2)--(1.5,1.5);
\fill[](1,2) node [right]{$(1,2)$} circle (4pt);
\end{tikzpicture}
\caption{The median for $z\leq4$} \label{fig:mediannotunique1}
\end{subfigure}
\qquad
\begin{subfigure}{0.45\textwidth}
\begin{tikzpicture}[scale=0.5]
\draw[gray, ->] (0,-1)--(0,7);
\draw[gray, ->](-2,0)--(7,0);
\draw[gray](-1,-1)--(7,7);

%

\draw plot[mark=triangle*,mark options={mark size=6pt, fill=white}]coordinates{
(0,2)} node [below left]  {$(0,2)$};
\draw plot[mark=triangle*,mark options={mark size=6pt, fill=white}]coordinates{
(3,5)} node [above]  {$(3,5)$};

\draw plot[mark=square*,mark options={mark size=4pt, fill=white}]coordinates{
(1, 4.5)} node [above]  {$(1,z)$};

\draw[dashed] (0,2)--(1,1);
\draw[dashed](3,4.5)--(1,4.5);

\draw[dashed] (3,5)--(3,4.5);
\draw [dashed] (3,4.5)--(3.75,3.75);

\fill[](3,4.5) node [below left]{$(3,z)$} circle (4pt);
\end{tikzpicture}
\caption{The median for $z\geq4$}\label{fig:mediannotunique2}
\end{subfigure}
\caption{
We have three diagrams, one consists only of copies of the diagonal, one containing off diagonal points denoted by squares, and the other containing off diagonal points denoted by  triangles.  When $z<4$ the optimal grouping is $\{(0,2), (1,z), \diag \}$ and $\{(3,5), \diag, \diag\}$ (the grouping used in (a)). When $z>4$ then the optimal grouping is $\{(0,2), \diag, \diag\}$ and $\{(3,5),  (1,z), \diag\}$  (the grouping used in (b)).  Both are optimal when $z=4$ and as a result we do not have a unique median.}\label{fig:MedianNotCont}
\end{figure}

%
%

The mean is generically unique but the median is not. To show this rigorously we shall restrict ourselves to the case where we have $N$ diagrams each with only finitely many off diagonal points. Let $k_1, k_2, \ldots k_N$ be non-negative integers.  Let $U(k_1, k_2, \ldots, k_N)$ denote the space of sets of diagrams $X=\{X_1, X_2, \ldots X_N\}$ such that $X_i$ has $k_i$ off diagonal points. $U(k_1, k_2, \ldots k_N)$ is the quotient of  $(\R^{2+})^{k_1 + k_2 + \ldots k_N}$ by a finite group of symmetries $\Gamma$. There is a quotient map 
\[q\colon (\R^{2+})^{k_1 + k_2 + \ldots k_N} \to U(k_1, k_2, \ldots k_N)= (\R^{2+})^{k_1 + k_2 + \ldots k_N} /\Gamma.\]
 Let $\lambda$ be Lebesgue measure on $(\R^{2+})^{k_1 + k_2 + \ldots k_N}$ and let $\rho=q_*(\lambda)$ be the push forward of Lebesgue measure onto $U(k_1, k_2, \ldots k_N)$.

\begin{proposition}
The sets of diagrams in $U(k_1, k_2, \ldots k_N)$ which do not have a unique mean has measure zero.
\end{proposition}
\begin{proof}
Let $\tilde{A}$ be the set of sets of diagrams in $U(k_1, k_2, \ldots k_N)$ which do not have a unique mean. Then $A=q^{-1}(\tilde{A})$ is the set of vectors of labelled diagrams (in $(\R^{2+})^{k_1 + k_2+ \ldots k_N}$) which do not have a unique mean.
Since $\rho(\tilde{A}) = \lambda(q^{-1}(\tilde{A}))$ it is sufficient to show $\lambda(A)=0$.


Let $S$ be a selection containing the points $\{(a_1,b_1), (a_2, b_2), \ldots (a_k,b_k)\}$ with $N-k$ copies of the diagonal. 
In the appendix we define the mean of the selection $S$ as the minimizer of $f_S(y)= \sum_{x\in S} \|x-y\|_2^2.$ which occurs at
$$\mu_S=\left(\frac{1}{N}\left(k\hat{x} + (N-k)\frac{\hat{x}+\hat{y}}{2}\right),\,\frac{1}{N}\left({k\hat{y} + (N-k)\frac{\hat{x}+\hat{y}}{2}}\right)\right)$$
where $\hat{x}$ and $\hat{y}$ are the means of $a_1, a_2, \ldots a_k$ and $b_1, b_2, \ldots b_k$ respectively. 


For each pair of distinct groupings, $G_1$ and $G_2$, of the labelled diagrams let $A(G_1, G_2)=\left\{X=(X_1, X_2, \ldots X_N)\colon \sum_{S\in G_1} f_S(\mu_S)=\sum_{S\in G_2} f_S(\mu_S)\right\}.$
Any $X\in A(G_1, G_2)$ must satisfy a quadratic equation so either $A(G_1, G_2)=(\R^{2+})^{k_1+k_2+\ldots k_N}$ or $\lambda(A(G_1, G_2))=0$. It is clear that there exists a vector of labelled persistence diagrams $X=(X_1, X_2, \ldots X_N) \in (\R^{2+})^{k_1+k_2+\ldots k_N}$ such that $X\notin A(G_1,G_2)$ we conclude that $\lambda(A(G_1, G_2))=0$. 

If $X$ has more than one mean then by Proposition \ref{thm:meannecc} there must be groupings $G_1,G_2$ such that $\mu_{G_1}\neq \mu_{G_2}$ but
$\sum_{S\in G_1}f_S(\mu_S)=F_2(\mu_{G_1})=F_2(\mu_{G_2})=\sum_{S\in G_2}f_S(\mu_S).$
This implies $A \subseteq \bigcup_{G_1\neq G_2  \,\text{groupings}} A(G_1, G_2).$ There are only finitely many groupings so $\lambda (A) =0$.
\end{proof}

This proof of generic uniqueness contrasts sharply to the case of the median which is not generically unique.

\begin{proposition} Let $N\geq 3$ be an odd number. 
Let $k_1, k_2, \ldots ,k_{(N+1)/2}\geq 2$. The sets of diagrams in $U(k_1, k_2, \ldots k_N)$ which do not have a unique median has positive measure.
\end{proposition}
\begin{proof}
%
%
We will first illustrate this with the case $U(2,2,0)$ which shows the idea of the general case. Suppose $X_1$  and $X_2$ each contain two off diagonal points $\{(a_1, a_2),(a_2, b_2)\}$, and  $\{(c_1, d_1) , (c_2, d_2)\}$ respectively, and $X_3$ has no off diagonal points. Further suppose that 
$a_1, a_2 < c_1, c_2 \leq b_1, b_2 <d_1, d_2.$ 

First consider the grouping $G_1=\{S_{(1,2)}, S_{(2,1)}\}$ where  $S_{(1,2)}:=\{(a_1, b_1), (c_2, d_2) ,\diag\}$ and $S_{(2,1)}:=\{(a_2, b_2), (c_1, d_1) ,\diag\}$. The median of the selection $S_{(1,2)}$ is $(c_2, b_1)$ and the median of the selection $S_{(2,1)}$ is  $(c_1, b_2)$. This implies that  $m_{G_1}$ has off-diagonal points $\{(c_2, b_1), (c_1, b_2)\}$. Also consider the grouping $G_1=\{S_{(1,2)}, S_{(2,1)}\}$ where  $S_{(1,2)}:=\{(a_1, b_1), (c_2, d_2) ,\diag\}$ and $S_{(2,1)}:=\{(a_2, b_2), (c_1, d_1) ,\diag\}$. Analogous calculations show the off-diagonal points of $m_{G_2}$ are $\{(c_1, b_1), (c_2, b_2)\}$. These groupings are illustrated in Figures \ref{fig:reallynotunique1} and \ref{fig:reallynotunique2}.


\begin{figure}[h]
\begin{subfigure}{0.45\textwidth}
\begin{tikzpicture}[scale=0.5]
\draw[gray, ->](-1,0)--(8,0);
\draw[gray, ->](0,-1)--(0,8);
\draw[gray](-1,-1)--(8,8);

\draw plot[mark=triangle*,mark options={mark size=6pt, fill=white}]coordinates{
(-0.5,5.5) } node [above left]  {$(a_1, b_1)$};
\draw plot[mark=square*,mark options={mark size=4pt, fill=white}]coordinates{
 (2.5 , 7.5)} node [above left]  {$(c_1, d_1)$};

\draw plot[mark=triangle*,mark options={mark size=6pt, fill=white}]coordinates{
 (0.5,5) } node [below left]  {$(a_2, b_2)$};
\draw plot[mark=square*,mark options={mark size=4pt, fill=white}]coordinates{
 (3 , 7)} node [above right]  {$(c_2, d_2)$};
%

\draw[dashed] (-1,5.5)--(8,5.5);
\draw[dashed] (-1,5)--(8,5);
\draw[dashed](2.5,-1)--(2.5,8);
\draw[dashed] (3,-1)--(3,8);

\fill[] (2.5,5) node[below left]{$(c_1, b_2)$} circle (4pt);
\fill[] (3,5.5) node[above right]{$(c_2, b_1)$} circle (4pt);

\end{tikzpicture}
\caption{$m_{G_1}$}\label{fig:reallynotunique1}
\end{subfigure}
\begin{subfigure}{0.45\textwidth}\begin{tikzpicture}[scale=0.5]
\draw[gray,->](-1,0)--(8,0);
\draw[gray,->](0,-1)--(0,8);
\draw[gray](-1,-1)--(8,8);

\draw plot[mark=triangle*,mark options={mark size=6pt, fill=white}]coordinates{
(-0.5,5.5) } node [above left]  {$(a_1, b_1)$};
\draw plot[mark=square*,mark options={mark size=4pt, fill=white}]coordinates{
 (2.5 , 7.5)} node [above left]  {$(c_1, d_1)$};

\draw plot[mark=triangle*,mark options={mark size=6pt, fill=white}]coordinates{
 (0.5,5) } node [below left]  {$(a_2, b_2)$};
\draw plot[mark=square*,mark options={mark size=4pt, fill=white}]coordinates{
 (3 , 7)} node [above right]  {$(c_2, d_2)$};
 
%

\draw[dashed] (-1,5.5)--(8,5.5);
\draw[dashed] (-1,5)--(8,5);
\draw[dashed](2.5,-1)--(2.5,8);
\draw[dashed] (3,-1)--(3,8);

\fill[] (2.5,5.5) node[above left]{$(c_1, b_1)$} circle (4pt);
\fill[] (3,5) node[below right]{$(c_2, b_2)$} circle (4pt);

\end{tikzpicture}
\caption{$m_{G_2}$}\label{fig:reallynotunique2}
\end{subfigure}
\begin{subfigure}{0.8\textwidth}
\begin{center}
\begin{tikzpicture}[scale=0.6]
\draw[->](-1,0)--(10,0);
\draw[->](0,-1)--(0,10);
\draw[](-1,-1)--(10,10);

\fill[lightgray] (1, 4) rectangle (2, 5);
\node at (1.5,4.5) {$1$};
\fill[lightgray] (2,5) rectangle (3,6);
\node at (2.5,5.5) {$2$};
\fill[lightgray] (0,6) rectangle (1,7);
\node at (0.5,6.5) {$3$};
\fill[pattern=crosshatch] (7.5,7.5)--(7.5,10)--(10,10);
\end{tikzpicture}
\end{center}
\caption{The different regions of $\R^{2+}$ for constructing sets of populations of persistence diagrams with non unique medians.}\label{fig:generallynotunique}
\end{subfigure}
\end{figure}

Now $F_1(m_{G_1})=-a_1 + d_1 -a_2 + d_2=F_1(m_{G_2})$ and that $F_1(m_{G})\geq -a_1 + d_1 -a_2 + d_2$ for all other groupings $G$. This implies that $m_{G_1}$ and $m_{G_2}$ are both medians of $X$. If $b_1\neq b_2$ and $c_1\neq c_2$ these medians are distinct and thus we do not have a unique median. The measure of such sets of diagrams $\{X_1, X_2, X_3\}$ has non-zero measure in $U(2,2,0)$.
%
%

The extension of this example to when $k_1, k_2, \ldots k_{(N+1)/2}>2$ is illustrated in Figure \ref{fig:generallynotunique}. We need to find an example of a non-zero measure set of $(X_1, X_2, \ldots X_N) \in (\R^{2+})^{k_1+ k_2 + \ldots + k_N}$ with $k_1, k_2, \ldots k_{(N+1)/2}>2$  with non unique medians. 

We will require that: 
\begin{itemize}
\item $X_1$ has two points $(a_1, b_1)$ and $(a_2, b_2)$ in region $1$,
\item $X_2$ has two points $(c_1, d_1)$ and $(c_2, d_2)$ in region $2$,
\item $X_3, X_4 \ldots X_{(N-3)/2}$  each contains two points in region $3$, and 
\item every other off-diagonal point in the $X_i$ lies in the region patterned by crosshatch.
\end{itemize}
Note that this set of populations of persistence diagrams is of non-zero measure in $U(k_1, k_2, \ldots k_N)$.

Every median of $\{X_i\}$ can be written as $m_G$ where each selection in $G$ contains either points in the cross-hatch region (and potentially copies of the diagonal) or they contain one point each from regions $1$ and $2$, $(N-3)/2$ points from region $3$, and $(N-1)/2$ copies of the diagonal. There is a is a median $m$ with off diagonal points $\{(c_1, b_1), (c_2, b_2)\}$ alongside other points determined by the points in the cross-hatched region. Another median $\tilde{m}$ is the same as $m$ but switching  $\{(c_1, b_1), (c_2, b_2)\}$ for  $\{(c_2, b_1), (c_1, b_2)\}$.

\end{proof}


\section{Discussion and further directions}
There are many parallels between the mean and median of populations of persistence diagrams. This suggests some future directions could involve extending work that has been done on the mean to the corresponding results for the median. For example, in \cite{munch2015probabilistic} they explore an alternative probabilistic definition of the mean which combines the tradition mean used with the notion of a shaking hand equilibrium in game theory. This alternate definition is unique and continuous. We believe a similar idea would work to create a probabilistic definition of the median.
%

Another future direction is to combine the median with sampling theorems to find conditions to infer the correct homology with high probability. For example, the homology of a set can be inferred from the persistence diagram corresponding to a point cloud with small Hausdorff distance to the original set. Under certain sampling conditions we can ensure that this Hausdorff distance is small with high probability. Perhaps with higher probability the median of independently obtained persistence diagrams under such samplings conditions will provide the correct homology.

There is scope for further developments in algorithms, both in design and implementation. In this paper we have discussed very referred to some of the computational aspects, including mentioning a brute force algorithm and a gradient descent approach. Perhaps there could be significant improvements by using geometry analogous to the work in \cite{kerber2017geometry} where they show that by exploiting the inherit geometry of the points in persistence diagrams lying in a plane we can approximate the Wasserstein distances much faster. 

\bibliographystyle{plain}	
\bibliography{standardised_bib}

\appendix
\section{Rips filtration}

The Rips filtration $\mathcal{R}$ is the filtration of the flag complex on $25$ vertices where at time $t$, $\mathcal{R}_t$ contains all the vertices $[v]$, the edges $[v_0, v_1]$ whenever $\|v_0-v_1\|<t$, the 2-simplicies $[v_0,v_1,v_2]$ whenever $[v_0,v_1]$, $[v_1, v_2]$ and $[v_2, v_0]$ are all included in $\mathcal{R}_t$, and so on including higher dimensional simplicies whenever all their boundary faces are in $\mathcal{R}_t$.

In other words $\mathcal{R}_t$ is the flag complex  (also known as the clique complex) on the graph containing all edges of length at most $t$.

\section{Mean diagram} 

The methods here provide a proof for the necessary condition for a persistence diagram to be a local minimum of of the Fr\'echet function which is far simpler than that in \cite{turner2014frechet}. We also extend the results to persistence diagrams containing points in $\L_\infty$ and $\L_{-\infty}$. Due to the similarities to the earlier material we omit many of the details.

For the mean we also split our analysis into the restrictions to $\R^{2+}\cup\diag$, $\L_{\infty}$ and $\L_{-\infty}$. If $X_1, X_2, \ldots, X_N\in \D^{(k,l)}$ then $Y$ is a mean of the $X_i$ if and only if $Y|_{\R^{2+}\cup \diag}, Y|_{\L_{\infty}}$, and  $Y|_{\L_{-\infty}}$ are means of the $X_i$ each restricted to the appropriate domain. 

As $\L_{\infty}$ and $\L_{-\infty}$ are effectively copies of $\R$ we can easily characterize the means of populations of multisets in them. Suppose $A_1, A_2, \ldots , A_N$ are each multisets of exactly $k$ real numbers and we label the elements of each $A_i$ so that $A_i = \{a_{i,1}, a_{i,2}, \ldots, a_{i,k}\}$ with $a_{i,1} \leq a_{i,2} \leq \ldots \leq a_{i,k}$. Set $B= \{b_1, b_2, \ldots, b_k\}$ where $b_j$  is the mean of $\{a_{1,j}, \ldots a_{N,j}\}$. Then $B$ is the unique multiset of $k$ real numbers that minimizes
\[f_2\colon Y \mapsto \sum_{i=1}^N \left( \inf_{\phi\colon A_i \to Y, \phi \text{ bijection}} \sum_{a\in A_i} |a-\phi (a)|^2\right)\]  and hence $B$ is the mean.

Characterizing the means of populations of persistence diagrams in $\D^{(0,0)}$ is analogously achieved through the means of selections.

\begin{lemma}\label{lemma:meanselection}
Let $(a_1, b_1), (a_2, b_2), \ldots, (a_k, b_k)$ be points in the plane. Let $\hat{x}$ be the mean of $a_1, a_2 \ldots a_k$ and $\hat{y}$ be the mean of $b_1, b_2, \ldots , b_k$. Then 
\[(\tilde{x},\tilde{y})\colon =\left(\frac{1}{N}\left(k\hat{x} + (N-k)\frac{\hat{x}+\hat{y}}{2}\right),\,\frac{1}{N}\left( k\hat{y} + (N-k)\frac{\hat{x}+\hat{y}}{2}\right)\right)\]
 is the unique point in $\R^{2+}$  which minimizes
\[f(x,y) =\sum_{i=1}^k \|(x,y)-(a_i,b_i)\|_2^2+ \sum_{i=k+1}^N\|(x,y)-\diag\|_2^2.\]
\end{lemma}

Let $S$ be a mulitset in $\R^{2+}\cup \diag$ containing $(a_1, b_1), (a_2, b_2), \ldots, (a_k, b_k)$ and $N-k$ copies of the diagonal and let $(x,y)$ be the point in $\R^{2+}$ found in Lemma \ref{lemma:meanselection}. We call this $(x,y)$ the \emph{mean} of $S$ and denote it by $\mu_S$.

Let $\mu(G)$ be the persistence diagram which contains $\{\mu_{S_j}:S_j\in G\}$. Each grouping $G$ produces a candidate $\mu(G)$ for the mean. We will show that any mean must be $\mu(G)$ for some grouping $G$.
\begin{figure}[hbt]
\begin{center}
\begin{tikzpicture}[scale=.8]
\draw[->] (0,-1)--(0,7);
\draw[->](-1,0)--(7,0);
\draw[](-1,-1)--(7,7);

\draw (1,4) circle (4pt);
\draw[dashed](1,4)--(1.7,4.2) ;
\draw[dashed](2.4,5)--(1.7,4.2);
\draw[dashed](1.7,3.6)--(1.7,4.2);
\draw (2.4,5) circle (4pt);
\draw (1.7,3.6) circle (4pt);
\draw plot[mark=square*, mark options={mark size=3.7pt, fill=gray}] coordinates{
(1.7,4.2)};

\draw plot[mark=diamond*,mark options={mark size=5pt, fill=white}] coordinates{
(2.95,2.95)};

\draw[dashed](2.95,2.95)--(1.7,4.2);

\draw plot[mark=triangle*,mark size=5pt] coordinates{
(2.2,3.7)};

\end{tikzpicture}
\caption{We want the mean of the three points marked by circles alongside two copies of the diagonal. The gray square is the arithmetic mean of the three points marked by circles. The diamond is the point on the diagonal closest to the square. The triangle is the mean of the circles and $2$ copies of the diamond. It is the weighted average of the square and the diamond.}
\end{center}
\end{figure}
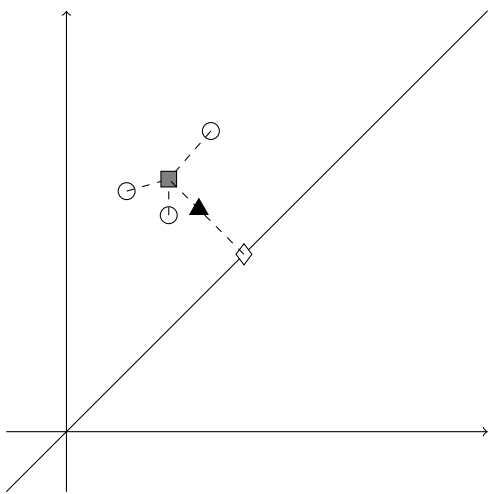

\begin{theorem}\label{thm:meannecc}
Let $X_1, \ldots, X_N, Y\in \D^{(0,0)}$ be persistence diagrams each with finitely many off diagonal points. Let $F_2:Z \mapsto \frac1N \sum_{i=1}^N d_2(X_i,Z)^2$. For each $i$ fix an optimal bijection $\phi_i:Y \to X_i$ (for $d_2$). For each $y\in Y$ we have a selection $\{\phi_i(y)\}$. Let $G_Y$ be the grouping $\{ \{\phi_i(y)\}\colon y\in Y\}$. If $Y$ is a local minimum of $F_2$ then $Y=\mu_{G_Y}$.
\end{theorem}

\begin{proof}
Suppose that $Y\neq \mu_{G_Y}$ and thus $y_0\neq \mu_{\{\phi_i(y_0)\}}$ for some $y_0\in Y$.  Set $Y_t$ to be the diagram which agrees with $Y$ except the point $y_0$ is replaced with $(1-t)y_0 + t \mu_{\{\phi_i(y_0)\}}$. 
\begin{align*}
F_2(Y_t)&=\frac{1}{N}\sum_{i=1}^N \inf_{\phi_i^t:Y_t\to X_i} \sum_{y_t\in Y_t} \|y_t-\phi_i^t(y_t)\|^2\\
&\leq \frac{1}{N}\sum_{i=1}^N \sum_{y\in Y, y\neq y_0} \|y-\phi_i(y)\|^2 + (((1-t)y_0 + t \mu_{\{\phi_i(y_0)\}})-\phi_i(y_0))^2
\end{align*}
We thus can conclude that for all $t\in (0,1)$
\begin{align*}
F_2(Y_t)-F_2(Y_0) \leq  \frac{1}{N}\sum_{i=1}^N(y_0-\phi_i(y_0))^2 - (((1-t)y_0 + t \mu_{\{\phi_i(y_0)\}})-\phi_i(y_0))^2
\end{align*}
which we know is negative from the proof of Lemma \ref{lemma:meanselection}. This implies that $Y_0=Y$ can not be a local minimum. 
\end{proof}

\end{document}